\documentclass[reqno]{amsart}

\usepackage{amsmath}
\usepackage{amsfonts}
\usepackage{rotating}
\usepackage{bbm}
\usepackage[all,cmtip]{xy}
\usepackage{amscd}
\numberwithin{equation}{section}

\usepackage{cases}

 \usepackage[tight,footnotesize]{subfigure}
 
\usepackage{tikz}
\usetikzlibrary{arrows,through} 
\usepackage{mathrsfs}
\usepackage{color}
\usepackage{xcolor}
\usepackage{mathtools}
\usepackage{amssymb}
\usepackage[latin1]{inputenc}
\usepackage{graphicx}
\usepackage{dsfont}
\usepackage{color}
\usepackage{hyperref}
\usepackage{marginnote}

\usepackage{graphicx}

\newtheorem{theorem}{Theorem}[section]

\newtheorem{lemma}[theorem]{Lemma}
\newtheorem{proposition}[theorem]{Proposition}
\newtheorem{corollary}[theorem]{Corollary}

\theoremstyle{definition}
\newtheorem{definition}[theorem]{Definition}

\theoremstyle{remark}
\newtheorem{remark}[theorem]{Remark}

\numberwithin{equation}{section}

\newcommand{\ow}{\omega}
\newcommand{\p}{\partial}

\DeclareMathOperator{\RFH}{RFH}
\DeclareMathOperator{\FC}{FC}
\DeclareMathOperator{\FH}{FH}

\usepackage{tkz-euclide,subfigure}
\usetkzobj{all}

\newcommand{\R}{{\mathbb{R}}}

\newcommand{\Z}{{\mathbb{Z}}}
\newcommand{\N}{{\mathbb{N}}}

\renewcommand{\epsilon}{\varepsilon}
\renewcommand{\theta}{\vartheta}

\DeclareMathOperator{\crit}{crit}
\DeclareMathOperator{\loc}{loc}

\DeclareMathOperator{\RS}{RS}

\DeclareMathOperator{\im}{im}

\usepackage{graphicx}

\begin{document}
\title[Bifurcations of symmetric periodic orbits via Floer homology]{Bifurcations of symmetric periodic orbits via Floer homology}
\author{Joontae Kim, Seongchan Kim and Myeonggi Kwon}
 \address{School of Mathematics, Korea Institute for Advanced Study, 85 Hoegiro, Dongdaemun-gu, Seoul 02455, Republic of Korea}
 \email {joontae@kias.re.kr}

 \address{Institut de Math\'ematiques, Universit\'e de Neuch\^atel, Rue Emile-Argand 11, 2000 Neuch\^atel, Switzerland}
 \email {seongchan.kim@unine.ch}

  \address{Mathematisches Institut, Justus-Liebig-Universit\"at Gie\ss en, Arndtstra\ss e 2, 35398 Gie\ss en, Germany}
 \email {Kwon.Myeonggi@math.uni-giessen.de}
%\date{Recent modification; \today}
%\setcounter{secnumdepth}{1}

\date{Recent modification; \today}
\setcounter{tocdepth}{2}
\maketitle
\begin{abstract}
 We give criteria for the existence of bifurcations of   symmetric  periodic orbits in reversible Hamiltonian systems in terms of local equivariant Lagrangian Rabinowitz Floer homology.  As an example, we consider the family of the direct circular orbits in the rotating Kepler problem and observe bifurcations of torus-type orbits. Our setup is motivated by   numerical work of H\'enon on Hill's lunar problem. 
\end{abstract}
\tableofcontents

\section{Introduction}

The study of periodic orbits  gives  crucial insight into dynamical systems as Poincar\'e regarded them as   the ``skeleton" of dynamical systems. An interesting behavior of periodic orbits can occur when periodic orbits appear in one-parameter families. Namely, under certain circumstances, the family of periodic orbits loses their stability and gives birth to new periodic orbits as the parameter varies. Such a phenomenon is called \emph{bifurcation}. Specifically, we are interested in bifurcations of symmetric periodic orbits in reversible Hamiltonian systems, inspired by H\'enon's numerical work on Hill's lunar problem. %Here we briefly recall his results. %Next to the Sun, the Moon shines most brightly among the objects in the sky and hence  it is very natural that people are interested in the motion of the Moon.  

 In 1878, Hill gave a new formulation of the lunar theory that he obtained from the planar circular restricted three-body problem \cite{Hill}:  Since the Sun is much heavier than the Earth and the Moon, and since the distance between   Sun and Earth is much bigger than that between   Earth and Moon, we may blow up coordinates near the Earth in the planar circular restricted three-body problem and obtain the Hamiltonian of Hill's lunar problem  $F: T^* ( \R^2 \setminus \left \{ 0 \right \}) \rightarrow \R$
$$
F(q,p) = \frac{1}{2} |p|^2 - \frac{1}{|q|} + q_1 p_2 - q_2 p_1 - q_1^2 + \frac{1}{2}q_2^2 .
$$
For more details on Hill's lunar problem, we refer the reader to the recent book \cite[Section 5.8]{book}. 
%The Hamiltonian $H$ is time-independent and hence is preserved along the Hamiltonian flow.   

An interesting feature of this problem is that $F$ admits the two commuting antisymplectic involutions 
\begin{equation}\label{eq:HIll}
\rho_1(q_1, q_2, p_1, p_2) = (q_1, -q_2, -p_1, p_2), \quad \rho_2 (q_1, q_2, p_1, p_2) = (-q_1 , q_2, p_1, -p_2).
\end{equation}
 In \cite{Henon}, by a numerical  exploration  H\'enon has found the family ``$g$"  of periodic orbits symmetric with respect to both involutions, whose    parameter is given by the energy $F=c$. This family converges to the direct circular orbit as $c \rightarrow - \infty$.  Increasing  the energy value from a very negative one, one finds an energy value $c=c_*$, which is close to  the unique critical value of $F$ and at which the family $g$ becomes degenerate. 
%where we denote by $g_*$ the corresponding degenerate orbit in the family. 
%For $c<c_*$, the $g$-orbits are elliptic and   for $c >c_*$ they are hyperbolic and the Maslov index changes by one, see for example  \cite{Cengiz}. 
At the moment $c=c_*$,  a pitchfork bifurcation happens, that is, two families of  (elliptic) periodic orbits bifurcate. H\'enon   observed that the two families are $\rho_1$-symmetric, but not $\rho_2$-symmetric, and that  they are related by the $\rho_2$-symmetry.  
This  is the so-called \emph{symmetry-breaking}.

Bifurcations of symmetric periodic orbits have been studied using different techniques, see for example  \cite{CV98, CV04, DE76, LT09, V86, V92, Kaz13}.      
In \cite{DX}, Deng and Xia study bifurcations by looking at generating functions of Hamiltonian diffeomorphisms whose critical points correspond to periodic orbits.

The purpose of the present  paper is to prove criteria for the occurrence of  bifurcations of symmetric periodic orbits in reversible Hamiltonian systems using Floer theory. 
While finding periodic  orbits
and chords of Hamiltonian systems and Reeb flows by Floer theory is a traditional topic, using Floer homology for the study of bifurcations is a fresh approach.
We use \emph{local equivariant Lagrangian Rabinowitz Floer homology}, which we abbreviate by \emph{local eLRFH}. Of particular importance in the argument is  the invariance property of local eLRFH under Hamiltonian perturbations.  We shall explain that given a one-parameter family of symmetric  periodic orbits satisfying certain conditions, if (a multiple cover of) an orbit in the family becomes degenerate, then the invariance of local eLRFH of the degenerate orbit forces the family to undergo a bifurcation, i.e.,   new  symmetric periodic orbits are born. 
 
 We should mention that the local bifurcation results under consideration can be obtained by more elementary approach using a finite dimensional reduction. We explain its idea in the case of periodic orbits. A similar construction for symmetric periodic orbits works as well. Consider a periodic orbit $\gamma$ of a Hamiltonian $H$ in a symplectic manifold $T^*N$ at which   bifurcation occurs. Its neighborhood is identified with $(\R/\Z\times \R^{2n-1},\sum_{i=1}^ndp_i\wedge dq_i)$, and   periodic orbits in this neighborhood naturally lift to loops in $\R^{2n}$. We consider a smooth nonlinear operator
 $$
 L\colon W^{1,2}(S^1, \R^{2n})\to L^2(S^1,\R^{2n}),\quad L(\gamma)=-J\dot{\gamma}+\nabla H(\gamma),
 $$
 whose solutions $L(\gamma)=0$
 correspond to \emph{smooth} periodic orbits by the Sobolev embedding theorem. Performing the saddle point reduction, see for instance \cite[Sections 2 and 12]{AmanZehnder}, we can transform this infinite dimensional problem into the problem of finding critical points of a functional defined on a finite dimensional space. See for example \cite[Proposition 2.1]{AmanZehnder} or \cite[Lemma 2.3]{CZ} for a detailed statement. On this reduced problem, critical points of the functional have well-defined Morse index, and hence the standard Morse theory can detect   bifurcation when we have a jump of the Morse index in a one-parameter family of critical points. 
 
For another approach, studying chords of (fiberwise convex) Hamiltonians on the cotangent bundle can be translated under the Fenchel transform into studying critical points of the free period Lagrangian action functional as in \cite[Proposition 2.2]{Asselle16}. See also \cite{Cont06}. In particular bifurcation phenomena could also be captured by an elementary Morse theoretical argument with respect to the Lagrangian action functional.

The possibility of an application of elementary methods does not weaken the value of our Floer theoretic approach, as it potentially opens up  possibilities to attack global bifurcation questions.

\section{Statement of the results}\label{sec:main}

\subsection{Setup}\label{sec: introsetup}
Let $N$ be a  connected $n$-dimensional smooth manifold equipped with an involution $f$. 
Assume that the fixed point set $Q:= \mathrm{Fix}(f) \subset N$ is non-empty and connected. It is well-known that $Q$ is a submanifold of $N$, see for instance \cite[Section 2.7]{book}.
Define the smooth map $\rho \colon T^*N \to T^*N$ as  
\[
\rho := d_* f \circ I = I \circ d_*f,
\]
where $d_*f \colon T^*N \to T^*N$ is the cotangent lift of $f$ and $I\colon T^*N \to T^*N, (q,p) \mapsto (q,-p)$.
It is easy to see that the map  $\rho$  is an exact anti-symplectic involution on $T^*N$, namely, it satisfies $\rho^* \lambda = -\lambda$, where $\lambda = pdq $ is the standard Liouville one-form on $T^*N$.
Its fixed point set equals the conormal bundle of $Q$
\[
 \mathrm{Fix}(\rho) = N^*Q := \{ (q,p) \in T^*N\mid q \in Q, p \in T_q^*N, \left< p, v\right> =0, \forall v \in T_q Q   \},
\]
which is a   connected  Lagrangian submanifold in $(T^*N, d\lambda)$, see for example \cite[Section 2.7]{book}.
In particular, $\lambda$ vanishes on $\mathrm{Fix}(\rho)$ which implies that the Lagrangian $L$ is exact.

Let $\Sigma \subset T^*N$ be a fiberwise starshaped hypersurface, meaning that  $\Sigma \cap T_q^*N$ bounds a smooth compact starshaped (with respect to the origin) domain in $T_q ^*N$ for every $ q \in N$.  
We assume that $\Sigma$ is invariant under $\rho $ and intersects   $N^*Q $.  Note that the intersection is transverse, and   the intersection $  \mathcal{L}= \Sigma \cap N^*Q$   is a Legendrian submanifold  in the contact manifold $(\Sigma, \alpha)$, where $\alpha = \lambda |_{\Sigma}$ is the restriction of $\lambda$ to   $\Sigma$.
 Let $F\colon T^*N \to \R$ be a Hamiltonian which is invariant under $\rho$ such that $\Sigma$ is a regular energy level set, say $\Sigma = F^{-1}(0)$.
In particular, every Hamiltonian orbit of $F$  on $\Sigma$ is a reparametrization of an   orbit of the Reeb vector field    $R  = R_{\alpha}$   associated to the contact form $\alpha$.

Let $(c,\eta)$ be a Reeb chord  with endpoints in $\mathcal{L}$, that is, $c \colon [0,\eta] \to \Sigma$ satisfies  
\[
   \dot{c}(t) = R(c(t)), \quad c(0), c(\eta) \in \mathcal{L}.
\]
It is easy to check that $\rho^* R = -R$, so that we have $c(\eta) = \rho ( c(0) )$.
In view of  the invariance of $\Sigma$ under  $\rho$  we obtain another Reeb chord $(\overline{c}  , \eta)$, defined by $\overline{c} (t):=\rho \circ c(\eta-t)$. 
An associated \textit{symmetric periodic orbit} $(x_c, 2 \eta)$ is a periodic Reeb orbit $x_c \colon [0,2\eta] \to \Sigma$ defined by
\[
x_c(t):= \begin{cases} c(t) & t \in [0,\eta], \\\overline{c} (t) & t \in [\eta,2\eta]. \end{cases}
\]
 Reeb chords come  in   pairs and there is a one-to-one correspondence between the sets of such pairs and of symmetric periodic orbits.
We call a Reeb chord $(c,\eta)$ \emph{simple} if the associated symmetric periodic orbit $(x_c , 2\eta)$ is simple.
If $(x_c, 2\eta)$ is the $n$-fold cover of $(x_{\mathfrak{c}}, 2\ell)$, that is $(x_c , 2\eta) = (x_{\mathfrak{c}}^n, 2n\ell)$, associated to a   simple Reeb chord  $(\mathfrak{c},\ell)$, then 
we say that $(c,\eta)$ is the $n$-fold cover of $(\mathfrak{c}, \ell)$ and write    $(c, \eta) = (\mathfrak{c}^n, n\ell)$.

\subsection{Bifurcations} \label{sec:bifurc}We now look at several bifurcation scenarios of symmetric periodic orbits. Let $F_\tau\colon T^*N \to \R$, $\tau\in(-\epsilon,\epsilon)$ be a smooth family of $\rho$-invariant Hamiltonians where $\epsilon > 0$  is   small.
 Assume that $0$ is a regular value of $F_{\tau}$ for every ${\tau}$.
Suppose further that every energy level  $\Sigma_{\tau}:=F_{\tau}^{-1}(0)$ is   fiberwise starshaped and transverse to $N^*Q$ and that $\Sigma_{\tau} \cap N^*Q \neq \emptyset$.  

Let $\{(c_{\tau}, \eta_{\tau}),(\bar{c}_\tau,\eta_{\tau})\}_{{\tau} \in (-\epsilon,\epsilon)}$ be a smooth family of  $\Z_2$-pairs of Reeb chords  such that  $\im(c_{\tau}) \subset\Sigma_\tau$ for every $\tau$. 
For $\tau_0 \in (-\epsilon, \epsilon)$, the sub-family  $\{ (  x_{c_{\tau}} ,2\eta_{\tau})\}_{{\tau} \in (-\epsilon,\tau_0]}$  of  the family $\{ (  x_{c_{ \tau}} ,2\eta_{\tau})\}_{{\tau} \in (-\epsilon,\epsilon)}$ of the associated $\rho$-symmetric periodic orbits is called    \emph{uniformly isolated}
if the   sub-family   $\{(c_{\tau}, \eta_{\tau}),(\bar{c}_\tau,\eta_{\tau})\}_{{\tau} \in (-\epsilon,\tau_0]}$ is uniformly isolated in the sense of Definition \ref{def: unifomlyisolated}. See also Figure \ref{Fig:unifiso}. 
This means that  we can take a  neighborhood $U=U_1 \times U_2 \subset T^*N \times \R$ of $\{(x_{c_{\tau}}, 2\eta_{\tau})\}_{\tau \in (-\epsilon, \tau_0]}$ 
such that for each $\tau \in (-\epsilon, \tau_0]$, $(x_{c_{\tau}}, 2\eta_{\tau})$ is the unique symmetric periodic orbit  on $\Sigma_{\tau}$ whose   image is contained in $U_1 \cap \Sigma_{\tau}$ and whose length is contained in $U_2$.

The first bifurcation scenario we look at is the following. Below, the index $\mu$ is the Robbin--Salamon index with a global shift, see \eqref{eq: maslovindex}.
\begin{theorem}\label{thm: intromainthm} Suppose that the smooth  family $\{ (  x_{c_{ \tau}} ,2\eta_{\tau})\}_{{\tau} \in (-\epsilon,\epsilon)}$ of   $\rho$-symmetric periodic orbits satisfies the following conditions.
\begin{itemize}
	\item The sub-family    $\{ (  x_{c_{ \tau}} ,2\eta_{\tau})\}_{{\tau} \in (-\epsilon,0]}$   is uniformly isolated.
	\item $(c_\tau,\eta_\tau)$ is non-degenerate for all $\tau\ne 0$.
	\item $(c_0,\eta_0)$ is degenerate, and at $\tau=0$ the   index $\mu(c_\tau,\eta_\tau)$ changes.
\end{itemize}
Then there exists $\delta \in (0, \epsilon)$ for which the following assertions hold.
\begin{enumerate}
\item
For each $\tau \in (0,\delta) $, we find another $\rho$-symmetric periodic orbit  $(x_{d_{\tau}}, 2\sigma_{\tau})$ such that $\im (x_{d_{\tau}}) \subset \Sigma_{\tau}$.
It tends to $ (x_{c_0},2\eta_0)  $ as $ \tau \to 0^+$.

\item 
If  $ (x_{c_\tau},2 \eta_\tau) = (x_{\mathfrak{c}_\tau}^n ,2 n \ell _\tau)$  for all $\tau \in (-\epsilon, \epsilon)$, where   $( {\mathfrak{c}}_\tau,  \ell_\tau)$'s are   simple Reeb chords such that     $ (\mathfrak{c}^m_\tau , m \ell _\tau) $ is non-degenerate for   every $\tau \in (-\epsilon, \epsilon)$ and   for every $m <n$, then the  $\rho$-symmetric periodic orbit $(x_{d_{\tau}}, 2\sigma_{\tau})$ obtained in $\mathrm{(1)}$  is simple for all $\tau \in (0, \delta)$.
\end{enumerate}
\noindent  
Moreover, if  we assume that all Reeb chords   with image in $\Sigma_{\tau} \cap U_1$ and with lengths in $U_2$ for $\tau \in ( 0, \delta)$ are non-degenerate, where  $U = U_1 \times U_2 \subset T^*N \times \R$ is an isolating neighborhood of  $\{ (  x_{c_{ \tau}} ,2\eta_{\tau})\}_{{\tau} \in (-\epsilon,0]}$,
and that $\Sigma_{\tau}$ is compact for all $\tau \in  ( 0, \delta)$, 
 then there exists $\delta_0 \in (0, \delta]$ such that there exist  two geometrically distinct one-parameter smooth families $\{ (x_{d_{\tau}}, 2\sigma_{\tau})\}_{\tau \in (0, \delta_0)}$ and $\{ (x_{d_{\tau}'}, 2\sigma_{\tau}')\}_{\tau \in (0, \delta_0)}$ of $\rho$-symmetric periodic orbits 
having the properties described in $\mathrm{(1)}$ and $\mathrm{(2)}$.
If $\mu( c_{\tau}, \eta_{\tau}) =k$ for $\tau <0$ and $\mu(c_{\tau}, \eta_{\tau}) = k+p$ with some non-zero integer $p$ for $\tau >0$, then for every $\tau \in (0, \delta_0)$ we have 
$\mu( d_{\tau}, \sigma_{\tau}) = k$ and $\mu( d_{\tau}' , \sigma_{\tau}' )  \in \{  k+p \pm 1  \}$.

%See Figure \ref{fig:theorem1}.
\end{theorem}
 \begin{figure}[h]
\begin{center}
\begin{tikzpicture}[scale=0.9]
\draw[thick] (-3,0) to [out=50, in=180] (-2,0.5);
\draw[thick] (-2,0.5) to [out=0, in=180] (1.4,0.2);
\draw[thick] (1.4,0.2) to [out=0, in=200] (3,0.4);
 %%%%%%%%%%%%%%
  \draw [fill] (0.5,0.24) circle [radius=0.08];
%%%%%%%%%%%%%%%%%
\draw[thick] (0.5, 0.24) to [out=70, in=185] (2.5,1);
\node at (-1.5, 0.9) {$ \small{\{(x_{c_\tau},2\eta_\tau) \}}$};
\node at (2.5, 1 )[right] {$ \small{\{(x_{d_\tau},2\sigma_\tau)\}}$};
\node at (2.5, -0.8)[right] {$ \small{\{(x_{d'_\tau},2\sigma'_\tau) \}}$};

\draw[thick] (0.5, 0.24) to [out=290, in=170] (2.5,-0.8);

\end{tikzpicture}
\end{center}
%\caption{An illustration of Theorem \ref{thm: intromainthm}}
\label{fig:theorem1}
\end{figure}

\noindent

 \begin{remark}  \label{rmk:index} We make the following remarks on the hypothesis of Theorem \ref{thm: intromainthm}.
 \begin{enumerate}
 \item  Since any two non-degenerate Reeb chords which are connected via a smooth family of non-degenerate Reeb chords must have the same Robbin-Salamon index, the nondegeneracy condition above implies that the index $\mu(c_{\tau}, \eta_{\tau})$ is constant on the interval  $(-\epsilon, 0)$ and the interval $(0, \epsilon)$, respectively.  
 \item The nondegeneracy of a Reeb chord $(c, \eta)$ is equivalent to the nondegeneracy of $(\overline{c},\eta)$.
 If they are non-degenerate, then their Robbin-Salamon indices coincide.
 \item By definition of the Robbin-Salamon index,   if the index changes by $m$, then we have $m \in [- n  +1,n-1]$, where $n=\dim N$,  see \cite{RSindex}.
 \end{enumerate}
\end{remark}
 
%\begin{remark} Due to the contact-type condition on $\Sigma_{\tau}$,  the period $\eta_{\tau}$ never blows up as $\tau$ converges to some value, i.e., a \textit{blue sky catastrophe} does not happen, see \cite[Section 7.6]{book}.
%This in particular implies that a homoclinic period blow-up does not occur, see \cite[Theorem 9.1]{DE76}.

%\end{remark}

\begin{remark}

  One might take a closer look at a family $\{x _{1,{\tau}}\}$ of symmetric periodic orbits bifurcating from an original family $\{ x_{0,{\tau}}\}$ and detect an index change along the associated family of Reeb chords. Proceeding in a similar way, one then finds a new family $\{x_{2,{\tau}}\}$ bifurcating from $\{x_{1,{\tau}}\}$.  Iterating this process one obtains a \textit{bifurcation tree} in a given reversible Hamiltonian system. 
Since we have used a local method, we are not able to study global properties of bifurcation trees. 
In \cite{GMVF81} Greene, MacKay, Vivaldi, and Feigenbaum provide a specific example for which they have found a period-doubling tree via a surface of section and describes the asymptotic behaviour of the tree.
A geometric self-similarity of a bifurcation tree in some reversible systems is studied by Lamb and Roberts in \cite{LR98}.

\end{remark}

As an application of Theorem \ref{thm: intromainthm}, we now look at bifurcations of symmetric periodic orbits from families of doubly symmetric periodic orbits.
Namely, we study symmetry-breaking which can be  observed in Hill's lunar problem as mentioned in the introduction.
 Assume that $F_{\tau}\colon T^*N \to \R$, $\tau\in(-\epsilon,\epsilon)$ is a smooth family of Hamiltonians invariant under  two commuting exact anti-symplectic involutions $\rho_1, \rho_2$ and set $\Sigma_{\tau}$ as before.
Denote by $\mathcal{L}_{j,\tau}$ the associated (non-empty) Legendrian in $\Sigma_{\tau}$ for $j=1,2$.
A Reeb chord from $\mathcal{L}_{1,\tau}$ to $\mathcal{L}_{2, \tau}$ is a Reeb orbit $c \colon [0, \eta] \to \Sigma_{\tau}$ meeting the boundary condition $c(0) \in \mathcal{L}_{1,\tau},  c(\eta) \in \mathcal{L}_{2,\tau}$.
Closing it up by means of $\rho_1, \rho_2$, we obtain a doubly symmetric periodic orbit of period $4\eta$
\[
y_c(t)  = \begin{cases}     c(t) & t \in [0, \eta] ,\\ \rho_2 \circ c(2\eta - t) & t \in [\eta, 2 \eta], \\ \rho_1 \circ \rho_2 \circ c(t- 2\eta) & t \in [2\eta, 3\eta] ,\\ \rho_1 \circ c (4\eta-t) & t \in [3\eta, 4\eta].  \end{cases}
\]

  \begin{corollary}\label{thm: intromain}
     Let $\{ (y_{{\tau}}, 4 \eta_{\tau})\}_{{\tau} \in (-\epsilon, \epsilon)}$ be a smooth family of doubly symmetric periodic orbits such that   $\im (y_{{\tau}} )\subset \Sigma_{\tau}$ for every $\tau$. 
We choose a parametrization of each $y_{ {\tau}}$ such that $y_{{\tau}}(0)\in \mathcal{L}_{1,\tau}$ (and hence $y_{{\tau}}(\eta_{\tau}) \in \mathcal{L}_{2,\tau}$)   and $\{y_{{\tau}}(0)\}_{\tau\in (-\epsilon, \epsilon)}$ is smooth.  
Suppose that
  \begin{itemize}
   \item    The sub-family $\{(y_{ { {\tau}}}, 4\eta_{\tau})  \} _{\tau \in (-\epsilon, 0]}$   is uniformly isolated.
 \item For every ${\tau} \in (-\epsilon, \epsilon)$, each Reeb chord $c_{1,{\tau}}  := y_{\tau} |_{[0, 2\eta_{\tau}]}$  is non-degenerate as a Reeb chord with endpoints in $\mathcal{L}_{1,{\tau}} $.
 \item The Reeb chord $c_{2,{\tau}}   := y_{\tau} |_{ [\eta_{\tau}, 3\eta_{\tau}]}$  is non-degenerate as a Reeb chord with endpoints in $\mathcal{L}_{2,{\tau}}$ for all $\tau \neq 0$,
 and  the  index $\mu( c_{2,{\tau}}, 2 \eta_{\tau})$ changes at $\tau=0$.
 \end{itemize}
 Then there exists $\delta \in (0, \epsilon)$ for which the following assertions hold.
\begin{enumerate}
\item
For each $\tau \in (0,\delta) $, we find two additional   $\rho_2$-symmetric  periodic orbits $ ( x_{d_{\tau}}, 2\sigma_{\tau} )  $ and $ ( \rho_1 \circ x_{d_{\tau} }, 2\sigma_{\tau}  )  $, but   not $\rho_1$-symmetric.  
They tend to $ (y_{ 0},4\eta_0)  $ as $ \tau \to 0^+$.

\item 
If  $ (c_{2, \tau},2 \eta_\tau) = (  \mathfrak{c}_\tau^n ,2 n \ell _\tau)$  for all $\tau \in (-\epsilon, \epsilon)$, where   $( {\mathfrak{c}}_\tau,  \ell_\tau)$'s are   simple Reeb chords with endpoints in $\mathcal{L}_{2,\tau}$ such that     $ (\mathfrak{c}^m_\tau , m \ell _\tau) $ is non-degenerate for   every $\tau \in (-\epsilon, \epsilon)$ and   for every $m <n$, then the  $\rho_2$-symmetric periodic orbits $(x_{d_{\tau}}, 2\sigma_{\tau})$ and $(\rho_1 \circ x_{d_{\tau}} , 2\sigma_{\tau} )$  are simple for all $\tau \in (0, \delta)$.

\end{enumerate}
\noindent  
We assume further that all Reeb chords   with image  in $\Sigma_{\tau} \cap U_1$  and with lengths in $U_2$ for $\tau  \in (0, \delta)$ are non-degenerate, where  $U = U_1 \times U_2 \subset T^*N \times \R$ is an isolating neighborhood of  $\{ (  y_{ { \tau}} ,4\eta_{\tau})\}_{{\tau} \in (-\epsilon,0]}$,
 that $\Sigma_{\tau}$ is compact for all $\tau \in (0, \delta)$, 
 and that $\mu( c_{2,\tau}, 2\eta_{\tau}) =k$ for $\tau <0$ and $\mu(c_{2,\tau}, 2\eta_{\tau}) = k+p$ with some non-zero integer $p$ for $\tau >0$.
 Then   there exists $\delta_0 \in (0, \delta]$ such that there exist   $2p$ smooth one-parameter families  $\{ (x_{d_{\tau}^j}, 2\sigma_{\tau}^j)\}_{\tau \in (0, \delta_0)}$,  $\{ (\rho_1 \circ x_{d_{\tau}^j}, 2\sigma_{\tau}^j)\}_{\tau \in (0, \delta_0)}$ for $j=0,1, \ldots, p-1$
 of   $\rho_2$-symmetric periodic orbits, but not $\rho_1$-symmetric, having the properties described in $\mathrm{(1)}$ and $\mathrm{(2)}$.
 For each $j$, we have $\mu(d_{\tau}^j, \sigma_{\tau}^j) = k +j$.

 \end{corollary}

   \begin{remark}\label{rmk:introRKP}
 One should be able to construct the local eLRFH associated to two commuting exact anti-symplectic involutions.
 In this case, one might look at bifurcations of doubly symmetric periodic orbits.
 This phenomenon is observed in the rotating Kepler problem, see Remark \ref{rmk:RKPdouble}. 
   \end{remark}

   \noindent
   Recall from the introduction that Hill's lunar problem carries the family of periodic orbits that are symmetric with respect to both $\rho_1$ and $\rho_2$ defined in \eqref{eq:HIll}.  
   A numerical experiment \cite{Ad20} shows that while the   associated half-chords with endpoints in $\mathrm{Fix}(\rho_1)$ are non-degenerate for all $c \in (c_*-\epsilon, c_*+\epsilon)$ for some $\epsilon>0$, the  ones  with endpoints in $\mathrm{Fix}(\rho_2)$  have index change at $c = c_*$.  Corollary \ref{thm: intromain} then tells us that  there exist two new families of $\rho_1$-symmetric periodic orbits that are permuted by $\rho_2$.  This is  precisely the same phenomenon as  H\'enon's numerical observation mentioned in the introduction.

   \smallskip
   
 To state the next scenario, we turn to the case of a single exact anti-symplectic involution $\rho$.

   \begin{theorem}\label{thm: intromainthm2}
 
 Suppose that a smooth family $\{ (  x_{c_{ \tau}} ,2\eta_{\tau})\}_{{\tau} \in (-\epsilon,\epsilon)}$ of   $\rho$-symmetric periodic orbits satisfies the following conditions.
\begin{itemize}
\item The Reeb chord $(c_0, \eta_0)$  is degenerate and all the other chords $(c_{\tau}, \eta_{\tau})$, $\tau \neq 0$, are non-degenerate. 
\item There exists another smooth family $ \{ (x_{c_{\tau}'}, 2\eta_{\tau}') \}_{\tau \in (-\epsilon, 0)}$ of non-degenerate  $\rho$-symmetric periodic orbits   which tends to  $ (x_{c_0}, 2 \eta_0)$ as $\tau \to 0^-$.
\item The sub-family $\{ (x_{c_{\tau} }, 2\eta_{\tau}), (x_{c_{\tau}'}, 2\eta'_{\tau}) \}_{\tau \in (-\epsilon, 0)} \cup \{ (x_{c_0}, 2\eta_0 )\} $ is uniformly isolated.  
\end{itemize}
Then there exists $\delta \in (0, \epsilon)$ satisfying that for each $\tau \in (0,\delta) $, we find another $\rho$-symmetric periodic orbit $(x_{d_{\tau}}, 2\sigma_{\tau})$ such that $\im (x_{d_{\tau}}) \subset \Sigma_{\tau}$.
It tends to $ (x_{c_0},2\eta_0)  $ as $ \tau \to 0^+$, and     assertion $\mathrm{(2)}$ of Theorem \ref{thm: intromainthm} holds.
Moreover, if  we assume that all Reeb chords   with image  in $\Sigma_{\tau} \cap U_1$ and with lengths in $U_2$ for $\tau \in (0,\delta)$ are non-degenerate, where  $U = U_1 \times U_2 \subset T^*N \times \R$ is an isolating neighborhood of   $\{ (x_{c_{\tau}}, 2\eta_{\tau}), (x_{c_{\tau}'}, 2\eta'_{\tau}) \}_{\tau \in (-\epsilon, 0]}$,
and that $\Sigma_{\tau}$ is compact for all $\tau \in (0, \delta)$, 
 then there exists $\delta_0 \in (0, \delta]$ such that there exists     a smooth one-parameter family $\{ (x_{d_{\tau}}, 2\sigma_{\tau})\}_{\tau \in (0, \delta_0)}$  of $\rho$-symmetric periodic orbits 
having the properties stated above.
 \end{theorem}
 \begin{figure}[h]
\begin{center}
\begin{tikzpicture}[scale=0.9]
\draw[thick] (-3,0.4) to [out=20, in=180] (-2,0.5);
\draw[thick] (-2,0.5) to [out=0, in=180] (1.4,0.2);
\draw[thick] (1.4,0.2) to [out=0, in=200] (3,0.4);
 %%%%%%%%%%%%%%
  \draw [fill] (0.5,0.24) circle [radius=0.08];
%%%%%%%%%%%%%%%%%
\draw[thick] (0.5, 0.24) to [out=70, in=185] (2.5,1);
\node at (-1.6,  0 ) {$ \small{\{(x_{c_\tau},2\eta_\tau)  \}}$};
\node at (2.2, 1.4) {$ \small{\{(x_{d_\tau},2\sigma_\tau) \}}$};
\node at (-2.2, 1.7) {$ \small{\{(x_{c'_\tau},2\eta'_\tau) \}}$};
\draw[thick] (0.5,0.24) to [out=120, in=0] (-3,1.2);

\end{tikzpicture}
\end{center}
%\caption{An illustration of Theorem \ref{thm: intromainthm2}}
\label{fig:theorem3}
\end{figure}

In the last scenario   we consider a family of symmetric periodic orbits that ends in a degenerate orbit.
\begin{theorem}\label{thm: intromainthm1} 
Let $\{ (x_{c_{\tau}}, 2\eta_{\tau}) \}_{\tau \in (-\epsilon, 0]}$ be a smooth family of $\rho$-symmetric periodic orbits satisfying    the following conditions.
\begin{itemize}
\item The  Reeb chord $(c_{\tau}, \eta_{\tau})$   is non-degenerate for every $\tau \in (-\epsilon, 0)$ and    $  (c_0, \eta_0)$ is degenerate.  
\item The    Reeb chord $(x_{c_0}, 2\eta_0)$    is isolated in the set of all  Reeb chords on $\Sigma_0$.
\item There exists an open neighborhood $V=V_1 \times V_2$ of $\im (x_{c_0}) \times \{2 \eta_0 \}$ in $T^*N \times \R$ such that for every $\tau>0$ sufficiently small, no symmetric periodic orbit  on $\Sigma_{\tau}$  has its traces in $V$.
\end{itemize}
Then there exists $\delta \in (0, \epsilon)$ such that  for each $\tau \in (-\delta,0) $, we find another $\rho$-symmetric periodic orbit $(x_{d_{\tau}}, 2\sigma_{\tau})$ such that $\im (x_{d_{\tau}}) \subset \Sigma_{\tau}$.
It tends to $ (x_{c_0},2\eta_0)  $ as $ \tau \to 0^-$.
Moreover, if  we assume that all Reeb chords   with image in $\Sigma_{\tau} \cap V_1$ and with lengths in $V_2$ for $\tau \in (-\delta, 0)$ are non-degenerate, and that $\Sigma_{\tau}$ is compact for all $\tau \in (-\delta, 0]$, 
 then there exists $\delta_0 \in (0, \delta]$ such that there exists    a smooth one-parameter family $\{ (x_{d_{\tau}}, 2\sigma_{\tau})\}_{\tau \in (- \delta_0,0)}$ such that $ (x_{d_{\tau}}, 2\sigma_{\tau})$ tends to $(x_{c_0}, 2\eta_0)$ as $\tau \to 0^-$.
 
 %see Figure \ref{fig:theorem2}.
\end{theorem}
 
 \begin{figure}[h]
\begin{center}
\begin{tikzpicture}[scale=0.9]

  \draw[thick] (-4.9, 1) to [out=0, in= 190](-1, 0.5);
\draw[thick]  ( -1, 0.5) to [in=90, out= 0](0, 0);
 
\draw[thick] (0,0) to [out=270, in=0] (-4, -0.5);
\draw[thick] (-4,-0.5) to [out=180, in=0] (-4.9, -0.5);
  \draw[fill] (0,0) circle [radius=0.05];
    \draw[fill, white] (4,0) circle [radius=0.05];

\node at (1,0.7) {$\{(x_{c_{\tau}}, 2\eta_{\tau}) \}$};
\node at (1,-0.85) {$\{(x_{d_{\tau}},  2 \sigma_{\tau} )\}$};

\end{tikzpicture}
\end{center}
%\caption{An illustration of Theorem \ref{thm: intromainthm1}}
\label{fig:theorem2}
\end{figure}

 \begin{remark}
 Theorem \ref{thm: intromainthm1}  holds in a more general situation: If   $\{ ({c_{\tau}},  \eta_{\tau}) \}_{\tau \in (-\epsilon, 0)}$ is   a family of non-degenerate Reeb chords such that the family does not extend over $0$ and  such that the $\omega$-limit set $\Omega \subset \Sigma_0$ is isolated in the set of Reeb chords on $\Sigma_0$, then we have the same conclusion as the theorem, see \cite[Theorem B]{BFvK19}.
 Note that Belbruno, Frauenfelder, and van Koert do not use the local eLFRH, but use an elementary method based on Floer's stretching method for time-dependent gradient flow lines, which does not require gluing.
 \end{remark}

\subsection{Local eLRFH} In Section \ref{sec:eLRH}, we construct  equivariant Lagrangian Rabinowitz Floer homology (eLRFH), based on ideas in \cite{AM18,  CF09,FS16RFH,Me14}. Basically, it is a Floer homology for the $\Z_2$-equivariant Rabinowitz action functional associated to a $\rho$-invariant defining Hamiltonian $F$ for a contact hypersurface $\Sigma$.  
%(the action functional is denoted by $\mathcal{A}^{F, N}$ in Section \ref{sec:eLRH}). 
\noindent
The critical points of the action functional correspond to $\rho$-symmetric periodic orbits in the contact hypersurface $\Sigma = F^{-1}(0)$, more precisely, $\Z_2$-pairs of Reeb chords $\{(c, \eta), (\overline c, \eta)\}$ where $\eta$ is the length of $c$. 

For an isolated set $\mathcal{F}$ of critical points  of the equivariant Rabinowitz action functional with  common action value, we can define its local equivariant Lagrangian Rabinowitz Floer homology (local eLRFH). It is defined along a standard scheme of Floer theory, see for instance  \cite{CFHW, GinzGurel,  Mclean}, by taking an isolating neighborhood $U$ of $\mathcal{F}$ in the sense of Section \ref{sec: localnonequi}.  Roughly speaking, we perform the same construction as for eLRFH, but only consider   generators and differentials  in $U$. A standard limiting argument in local Floer theory  stated in Lemma \ref{lem: localconveq}  is crucial.   We give a detailed description in Section \ref{sec:local}. 

In particular, if there is a degenerate chord in an isolated set of Reeb chords, then we need to perturb the Hamiltonian $F$ slightly in a neighborhood $U$ so that all Reeb chords in $U$ are non-degenerate. Proposition \ref{prop: invariance} asserts that the local eLRFH is independent of the choice of perturbations of the Hamiltonian. This invariance is the main ingredient of the proof of the  theorems.

\subsection{Outline of the proofs of the theorems} Let us outline the proof of Theorem \ref{thm: intromainthm} using the local LRFH. Consider a family of Reeb chords $\{ (c_{\tau}, \eta_{\tau})\}_{{\tau} \in (-\epsilon,\epsilon)}$ satisfying the hypothesis of Theorem \ref{thm: intromainthm}. The set $\mathcal{F} = \{(c_0, \eta_0), (\overline c_0, \eta_0)\}$ then forms an isolated set of Reeb chords. Since $(c_{\tau}, \eta_{\tau})$ is non-degenerate for $\tau \neq 0$, we can regard the Hamiltonian $F_{\delta}$ for $\delta < 0$ with $\lvert \delta \rvert$ small enough as a perturbation of $F_0$ to define $\FH_*^{\rm loc, \Z_2}(\widetilde{U};F_\delta)$ for $\mathcal{F}$ where $U$ is a uniformly isolating neighborhood for the sub-family $\{(c_\tau,\eta_\tau)\}_{\tau \in (-\epsilon, 0]}$ and $\widetilde{U}=U\times S^N$. 

In particular, $\FH_*^{\rm loc, \Z_2}(\widetilde{U}, F_{-\delta})$ for $\{(c_{-\delta}, \eta_{-\delta}), (\overline c_{-\delta}, \eta_{-\delta})\}$ with $\delta > 0$  sufficiently small is nontrivial only at the degree $* = \mu(c_{-\delta}, \eta_{-\delta})$ by Proposition \ref{prop: eRFHcomputation} (applied with the uniformly isolated condition on $\{(c_\tau,\eta_\tau)\}_{\tau \in (-\epsilon, 0]}$), that is,
\[
\FH_*^{\rm loc, \Z_2}(\widetilde{U}; F_{-\delta}) = \begin{cases} \Z_2 & \text{ if }* = \mu(c_{-\delta},\eta_{-\delta})  , \\ 0 & \text{ otherwise}.\end{cases}
\]
The invariance of $\FH_*^{\rm loc, \Z_2}(\widetilde{U}; F_\tau)$ (for small $|\tau|\ll 1 $) under Hamiltonian perturbations in Proposition \ref{prop: invariance} then implies that for each sufficiently small $\delta > 0$, the $\Z_2$-pair $\{(c_{\delta}, \eta_{\delta}), (\overline c_{\delta}, \eta_{\delta})\}$, and hence the corresponding symmetric periodic orbit, cannot be isolated inside $U$ since $\mu(c_{\delta}, \eta_{\delta}) \neq \mu(c_{-\delta}, \eta_{-\delta})$; otherwise, we could apply Proposition \ref{prop: eRFHcomputation} to $\FH_*^{\rm loc, \Z_2}(\widetilde{U}; F_{\delta})$ which results in a contradiction to the invariance. This means we must have another Reeb chord $(c'_{\delta}, \eta'_{\delta})$ in $\Sigma_{\delta}$   as asserted. 
The other two theorems are also proved   using the invariance of the local eLRFH. Detailed proofs are given in Section \ref{sec:proofs}.

\subsection{Example: the rotating Kepler problem} We provide an explicit example of a reversible Hamiltonian system, the rotating Kepler problem,  in which the first bifurcation scenario  described in    Theorem \ref{thm: intromainthm} occurs.
That such bifurcations occur in the rotating Kepler problem was already known before by computing all periodic orbits directly using elementary method,  see for instance   \cite{RKPCZ, book, KKJplus}.
 In this section, relying on  a variety of well-known facts on the dynamics, we  illustrate how our theorem applies to a concrete situation.

  The describing Hamiltonian $F \colon T^*(\R^2 \setminus \{ (0,0) \} )\to \R$  is given by 
 $$
F(q,p) = \frac{1}{2}|p|^2 - \frac{1}{|q|} + q_1 p_2 - q_2 p_1,
$$
which    has a unique critical value $ \tau= -3/2$. 
The symplectic form is   the standard one $\omega_0 = d p_1 \wedge dq_1 + dp_2 \wedge dq_2$ and the Liouville one-form is $\lambda_0 = p_1 dq_1 + p_2 dq_2$.
For $\tau <-3/2$, the regular energy level $F^{-1}(\tau)$ consists of a bounded component $\Sigma_{\tau}$ and an unbounded component. 
The bounded component $\Sigma_{\tau}$ is fiberwise starshaped  \cite{AFvKP12,  CFvK14Finsler} (actually, a stronger result holds: it is fiberwise convex).
The contact form on $\Sigma_{\tau}$ is the restriction of  $\lambda_0$.
The projection $\mathcal{K}_{\tau}$ of  $\Sigma_{\tau}$ under the footpoint projection $\pi\colon T^*( \R^2 \setminus \{ (0,0) \} ) \to \R^2 \setminus \{ ( 0,0) \}$ is the closed unit ball of radius $r=r(\tau)$ minus the origin,
where the radius $r \colon (- \infty, -3/2) \to (0,1)$ is  strictly increasing.

An interesting feature of this problem is that $F$ is invariant under the family of exact antisymplectic involutions
$$
\rho_{\theta} (q_1, q_2, p_1, p_2 ) := \begin{pmatrix}    \cos 2 \theta & \sin 2\theta & 0 & 0 \\ \sin 2\theta & - \cos 2\theta & 0 & 0 \\ 0 & 0 & - \cos 2\theta & - \sin 2\theta \\ 0 & 0 & - \sin 2\theta & \cos 2\theta        \end{pmatrix} \begin{pmatrix} q_1 \\ q_2 \\ p_1 \\ p_2 \end{pmatrix}
 $$
for   $\theta \in [0,  \pi]$. 
For the sake of simplicity,    we only consider 
\[
\rho_0 (q_1, q_2, p_1, p_2) = (q_1, -q_2, -p_1, p_2).
\]
The associated Lagrangian  is given by $L   = \mathrm{Fix}(\rho_0) = \{ q_2 = p_1 =0\}$, and    
the projection  $\ell   = \pi (L )$ is   the $q_1$-axis. 
 For each $\tau < -3/2$, the Lagrangian $L$ intersects transversally the energy level $\Sigma_{\tau}$ and the intersection  is non-empty.
 We set $\mathcal{L}_{  \tau} := L  \cap \Sigma_{\tau}$.

  For each $\tau < -3/2$, there exists a   circular orbit on $\Sigma_{\tau}$,  called the \textit{direct circular orbit}, 
 which   rotates in the same direction as the coordinate system.
 Obviously it is invariant under $\rho_0$.
Let $\{( x_{\tau}, 2\eta_{\tau})\}_{  {\tau} \in (-\infty, -3/2)}$ be the family of the simply covered direct circular orbits. We choose a parametrization of each $x_\tau$ such that $x_{\tau}(0)\in \mathcal{L}_\tau$ and the family $\{x_\tau(0)\}_{\tau\in (-\infty, -3/2)}$ is smooth.  
The $n$-fold cover of $(x_{\tau}, 2 \eta_{\tau})$ will be denoted by $(x_{\tau}^n, 2n\eta_{\tau})$.
Denote by $(c_{\tau}:=x_{\tau}|_{[0,\eta_{\tau}]}, \eta_{\tau})$ the Reeb chord with endpoints in $\mathcal{L}_{ \tau}$
and by $(c_{\tau}^n, n \eta_{\tau})$ the $n$-fold cover of $(c_{\tau}, \eta_{\tau})$ which is the (half) Reeb chord    associated to $(x_{\tau}^n, 2n\eta_{\tau})$.

Fix $N \geq 1$ and consider the family $\{ (c_{\tau}^N , N\eta_{\tau})\}_{\tau \in (-\infty, -3/2)}$ of $N$-covered Reeb chords. 
For any relatively prime integers $k,l \in \N$ with $k>l$ such that $k-l =N$, 
there exists a unique energy level 
\begin{equation}\label{eq:energy}
\tau_{k,l} := -\frac{1}{2} \bigg( \frac{k}{l} \bigg)^{\tfrac{2}{3}} - \bigg( \frac{l}{k} \bigg)^{\tfrac{1}{3}} \in \left(-\infty, -\frac{3}{2} \right) 
\end{equation}
such that  the   Reeb chord $(c_{{\tau}_{k,l} }^N, N\eta_{\tau_{k,l} })$ becomes degenerate and passing through it the Robbin-Salamon index jumps by one, see  \cite[Appendix B]{RKPCZ} and \cite[Section 7.1]{JLee}
(recall from Remark \ref{rmk:index} (3) that since we now have $n=2$, the index can only jump by one). 
 Since $\tau_{k,l}=\tau_{k, k-N}$ is strictly increasing and converges   to $-3/2$ as $k \to \infty$, the energy levels at which  $(c_{\tau}^N,N \eta_{\tau} )$ gives birth to   new Reeb chords are discrete.
Hence we  find $\epsilon = \epsilon(k,l)>0$ sufficiently small such that  the other Reeb chords $(c_{\tau}^N, N\eta_{\tau})$ are non-degenerate for $\tau \in ({\tau}_{k,l}  - \epsilon, {\tau}_{k,l} +\epsilon) \setminus \{ {\tau}_{k,l}  \}$. 
The subfamily $\{ (x_{\tau}^N, 2N\eta_{\tau})\}_{\tau \in ({\tau}_{k,l}  - \epsilon,{\tau}_{k,l}  ]}$ is uniformly isolated, see    \cite[Section 7]{RKPCZ} or \cite[Section 3.4]{KKJplus}.
By the assertion (1) of Theorem \ref{thm: intromainthm}, we find  additional $\rho_0$-symmetric (simply covered) periodic orbits $  y^{k,l}_{\tau}  $ for all $\tau>0$ small enough,   bifurcating from  $x_{{\tau}_{k,l} }^{N}$.
In fact, they are non-degenerate and  form a one-parameter family of $\rho_0$-symmetric periodic orbits, see  \cite[Section 8.4]{book} and  \cite{RKPCZ, KKJplus}.
 See also Figure \ref{Fig:rkp}.
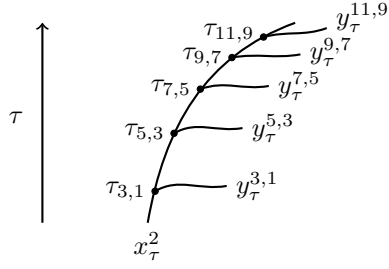
\begin{figure}[h]
\begin{tikzpicture}[scale=0.7]
\draw[thick] (-2,-2) [in=200, out=80] to (0.8,1.8);
\node at (-2, -2)[below] {$x_{\tau}^2$};
  \draw [fill] (-1.86,-1.4) circle [radius=0.06];
\draw[thick] (-1.86, -1.4) [out=30, in=190] to (-0.5, -1.3);
\node at (-0.5, -1.3)[right] {$y_{\tau}^{3,1}$};
\node at (-1.86, -1.4)[left] {$\tau_{3,1}$};

  \draw [fill] (-1.5 ,-0.3 ) circle [radius=0.06];
\draw[thick] (-1.5, -0.3) [out=30, in=190] to (-0.2, -0.2);
\node at (-0.2, -0.2)[right] {$y_{\tau}^{5,3}$};
\node at (-1.5, -0.3)[left] {$\tau_{5,3}$};

  \draw [fill] (-1  , 0.54  ) circle [radius=0.06];
\draw[thick] (-1 ,  0.54) [out=20, in=190] to ( 0.3,  0.6);
\node at (0.3, 0.6)[right] {$y_{\tau}^{7,5}$};
\node at (-1 , 0.54)[left] {$\tau_{7,5}$};

  \draw [fill] (-0.4, 1.14) circle [radius=0.06];
\draw[thick] (-0.4 ,  1.14) [out=10, in=190] to ( 0.9,  1.2);
\node at ( 0.9,  1.25)[right] {$y_{\tau}^{9,7}$};
\node at (-0.4, 1.14)[left] {$\tau_{9,7}$};

  \draw [fill] ( 0.2 , 1.53 ) circle [radius=0.06];
\draw[thick] (0.2 ,  1.53) [out=15, in=200] to ( 1.4,  1.7);
\node at (1.4,  1.9)[right] {$y_{\tau}^{11,9}$};
\node at (0.2,  1.6)[left] {$\tau_{11,9}$};

\draw[thick, ->] (-4,-2) to (-4,1.8);
\node at (-4.5, 0) {$\tau$};
\end{tikzpicture}
  \caption{An illustration of the bifurcations of $\rho_0$-symmetric  simple periodic orbits  from the doubly covered direct circular orbits.}
  \label{Fig:rkp}
\end{figure}
Therefore, we find another family of $\rho_0$-symmetric periodic orbits $\{ \rho_{\pi/2} \circ y_{\tau} ^{k,l} \}$ bifurcating from  $x_{{\tau}_{k,l} }^{N}$,  since $\rho_0$ and $\rho_{\pi/2}$ commute.
 Indeed, it is easy to check that  $\rho_{\theta_1} \circ \rho_{\theta_2} = \rho_{\theta_2 } \circ \rho_{\theta_1}$ if and only if $\theta_1 \equiv \theta_2 + \pi/2$  mod $\pi$.  
 Note that $\rho_{\pi/2} (q_1, q_2, p_1, p_2) = (-q_1, q_2, p_1, -p_2)$ and hence the projected images of $\{y_{\tau}^{k,l} \}$ and $\{ \rho_{\pi/2} \circ y_{\tau}^{k,l} \}$ to the $q$-plane are mapped to each other  by the reflection with respect to the $q_2$-axis. 
 See Figure   \ref{rkp2}.
   \begin{figure}[h]
    \includegraphics[width=0.7\linewidth]{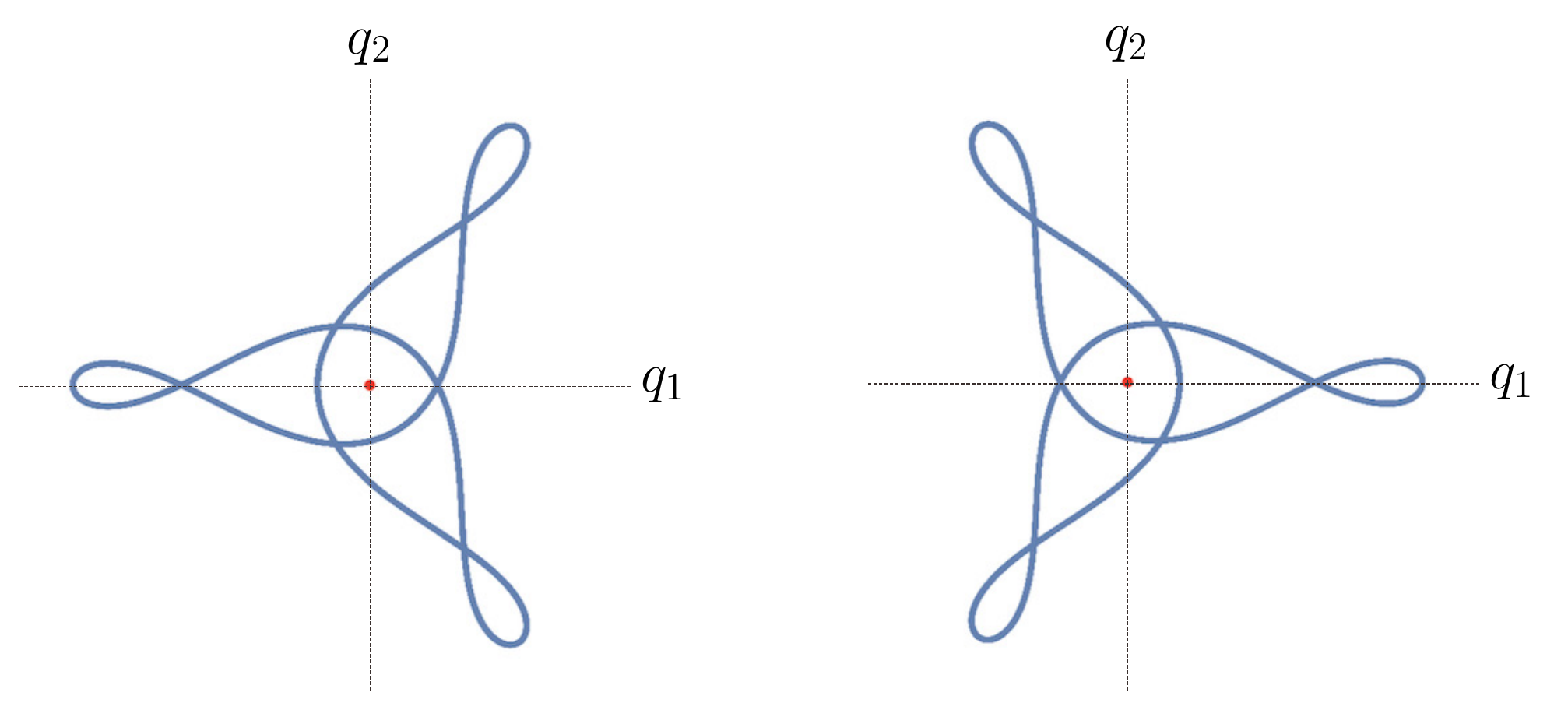}
\caption{The case $(k,l)=(3,1)$: The  images of two families of  $\rho_0$-symmetric  simple periodic orbits  under the footpoint projection $T^* \left( \R^2 \setminus \{ (0,0) \} \right) \to \R^2 \setminus \{ (0,0) \}$. They bifurcate from the degenerate doubly covered direct circular orbit and   are mapped to each other   by the reflection with respect to the vertical axis.}
\label{rkp2}
\end{figure}

 \begin{remark}
 The preceding discussion holds for every pair $\theta_1 \neq \theta_2 \in [0, \pi]$ such that $\theta_1 \equiv \theta_2 + \pi/2$ mod $\pi$.
 In this way, for each coprime $k>l$ with $N=k-l$, one finds an $S^1$-family of  symmetric periodic orbits, with prescribed initial conditions, bifurcating from the $N$-fold covered direct circular orbits $(x_{\tau}^N, 2N\eta_{\tau})$ at $\tau = \tau_{k,l}$.
Forgetting the symmetries, this provides a $T^2$-family of (parametrized) periodic orbits born out of   $(x_{\tau}^N, 2N\eta_{\tau})$ at $\tau = \tau_{k,l}$, which is  referred to as the $T_{k,l}$-family, see   \cite{RKPCZ,  KKJplus}.

 \end{remark}

 \begin{remark}\label{rmk:RKPdouble} 
As mentioned in Remark \ref{rmk:introRKP}, we now observe bifurcations of doubly symmetric periodic orbits in the rotating Kepler problem.
 Consider the family $\{(x_{\tau}, 2\eta_{\tau})\}_{\tau \in (-\infty, -3/2)}$ of the direct circular orbits as before,
 which are invariant under $\rho_0$ and $\rho_{\pi/2}$.
 We choose a parametrization of each $x_\tau$ such that $x_{\tau}(0)\in \mathcal{L}_{\tau,0}$ and $x_{\tau}(\eta_{\tau}/2) \in \mathcal{L}_{\tau, \pi/2}$,  where $\mathcal{L}_{\tau, \theta} = \mathrm{Fix}(\rho_{\theta}) \cap \Sigma_{\tau}$ for $\theta =0, \pi/2$.
Denote by   $\{(w_{\tau}:=x_{\tau} |_{[0, \eta_{\tau}/2]}, \eta_{\tau}/2)\}_{\tau \in (-\infty, -3/2)}$ the family of associated (quarter) Reeb chords.
 Fix $N \geq 1$ and consider the family $\{ (w_{\tau}^N , N\eta_{\tau} /2)\}_{\tau \in (-\infty, -3/2)}$. 
 Choose any relatively prime integers $k,l \in \N$ with $k>l$ such that $k-l =N$.
  It was shown in  \cite[Section 7.1]{JLee} that the  Robbin-Salamon index of $(w_{\tau}^N, N\eta_{\tau}/2)$ jumps by one if and only if
  \begin{equation}\label{eq:doubly}
  \tau = \tau_{k,l} \quad \text{ and } \quad \frac{k \left( 2(k-l)-1 \right)}{2(k-l)} \in \Z,
  \end{equation}
  where $\tau_{k,l}$ is defined as  \eqref{eq:energy}. 
  If these conditions are fulfilled, then a new family of doubly symmetric periodic orbits is born.
  For example, if $(k,l)=(3,1)$, then  bifurcations of symmetric periodic orbits occur at $\tau = \tau_{3,1}$ as above, but not bifurcations of doubly symmetric periodic orbits, since the last condition of    \eqref{eq:doubly} is not satisfied.
In the case that $(k,l)=(2,1)$, a new family $\{y_{\tau} \}$ of doubly symmetric periodic orbits is born out of the simple covered direct circular orbit  $\{(x_{\tau}, 2\eta_{\tau})\}$  at  $\tau = \tau_{2,1}$.
On the other hand,  if we consider the family   $\{(v_{\tau} := x_{\tau} |_{[\eta_{\tau}/2, \eta_{\tau}]}, \eta_{\tau}/2)\}_{\tau \in (-\infty, -3/2)}$, then 
we find another family $\{z_{\tau} \}$ of doubly symmetric periodic orbits born out of   $\{(x_{\tau}, 2\eta_{\tau})\}$ at $\tau = \tau_{1,2}$.
Note that the projected image  of  the two families $\{y_{\tau}\}$ and $\{z_{\tau}\}$ on $\R^2 \setminus \{ (0,0) \}$ are related by the $\frac{\pi}{2}$-rotation,  see Figure \ref{rkp3}.
      \begin{figure}[h]
    \includegraphics[width=0.7\linewidth]{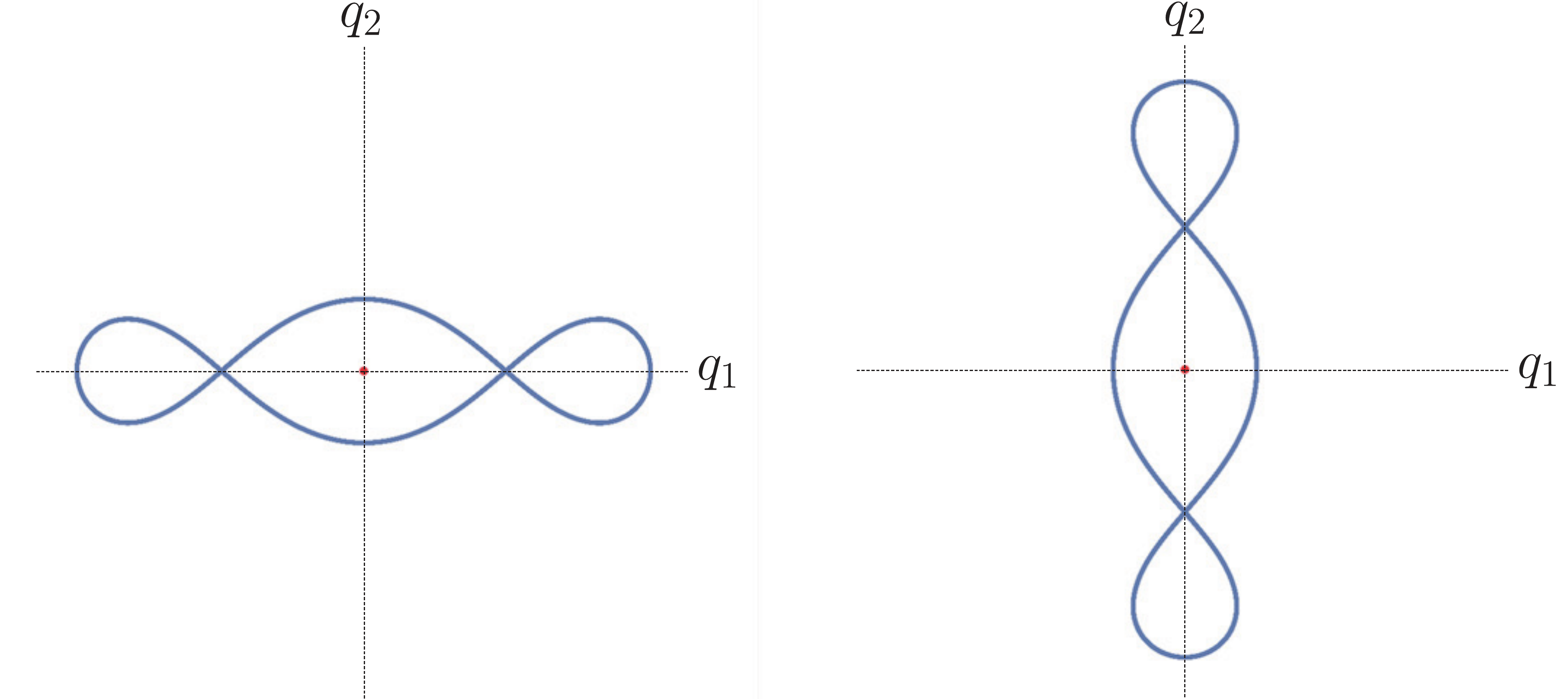}
\caption{The case $(k, l) = (2, 1)$: The  images of two families of   doubly symmetric periodic orbits  under the footpoint projection $T^* \left( \R^2 \setminus \{ (0,0) \} \right) \to \R^2 \setminus \{ (0,0) \}$.  They bifurcate from the degenerate simple covered direct circular orbit at $\tau = \tau_{2,1}$.}
\label{rkp3}
\end{figure}

 \end{remark}

 \bigskip

\noindent
\textbf{Acknowledgement.} The authors cordially thank Urs Frauenfelder for introducing to  them the paper of H\'enon and for helpful discussions on the invariance of the local equivariant Lagrangian Rabinowitz Floer homology and Felix Schlenk for reading a preliminary version very carefully.  Special thanks go to anonymous referee for valuable comments. JK is supported by Samsung Science and Technology Foundation under Project Number SSTF-BA1901-01. SK is supported by the grant 200021-181980/1 of the Swiss National Foundation. 
MK is supported  by SFB/TRR 191 ``Symplectic Structures in Geometry, Algebra and Dynamics" funded by the DFG.
\addtocontents{toc}{\protect\setcounter{tocdepth}{2}}

\section{Rabinowitz Floer theory}
\subsection{Equivariant Lagrangian Rabinowitz Floer homology}\label{sec:eLRH}
 Let $(W, \lambda)$ be the completion of a $2n$-dimensional Liouville domain $(W_0, \lambda_0)$.
 Abbreviate $\Sigma = \p W_0$.
 Suppose that $(W_0, \lambda_0)$ is equipped with an exact anti-symplectic involution $\rho_0$.
 The  hypersurface $\Sigma$ admits the contact form $\alpha := \lambda |_{\Sigma}$ and the anti-contact involution $\rho_{\Sigma}:= \rho_0|_{\Sigma}$, meaning that $\rho_{\Sigma}^*\alpha = -\alpha$.
 We extend $\rho_0$ to the exact anti-symplectic involution $\rho$ on $(W,\lambda)$ by
 \[
\rho(x) = \begin{cases} \rho_0(x)& x \in  W_0, \\ (r, \rho_{\Sigma}(y)) &  x = (r,y) \in [1, \infty) \times \Sigma . \end{cases}
\]
 Denote by $R = R_{\alpha}$ the associated Reeb vector field and by $\phi_R^t$ the Reeb flow. 
 We assume that the fixed point set $L = \mathrm{Fix}(\rho)$ of $\rho$ is non-empty and connected
 and that the intersection $\mathcal{L} = L \cap \Sigma$ is non-empty.
 It follows that $L$ is an exact Lagrangian submanifold of $(W, \lambda)$
 and that $\mathcal{L}$  is a Legrendrian submanifold of $(\Sigma, \alpha)$.

  A Reeb chord of length $\eta$ on $\Sigma $ is a pair $(c,\eta)$ such that $c\colon [0, \eta] \to \Sigma$ solves $\dot{c}(t) = R(c(t))$ for all $t\in[0,\eta]$ and   
  meets the boundary condition $c(0), c(\eta) \in \mathcal{L}$.
  It is called non-degenerate if it satisfies 
  \[
 D\phi_R^{\eta} ( T_{c(0)} \mathcal{L}) \cap T_{c(\eta)}\mathcal{L} = \{ 0 \}.
  \]
 In this section we assume that the contact form $\alpha$ is non-degenerate, that is, every Reeb chord is non-degenerate.

By \emph{a defining Hamiltonian for $\Sigma$} we mean a $\rho$-invariant Hamiltonian $F:W\to \R$ such that  $\Sigma=F^{-1}(0)$ is a regular energy level and $X_F|_\Sigma=R_\alpha$. Here the Hamiltonian vector field $X_F$ is defined by    $\ow(X_F,\cdot)=-dF$. Its flow is denoted by $\phi_F^t$.   Since $\Sigma$ splits $W$ into a bounded and an unbounded component, we can modify $F$ to be constant near infinity.

%\subsubsection{The equivariant Rabinowitz action functional}
Let $\mathscr{P}=\{c\in C^\infty([0,1],W) \mid c(0),c(1)\in L\}$ denote the space of smooth paths in $W$ with ends in $L$. We fix a Morse function $h_N$ on $S^N$ which is invariant under the antipodal map $ z \mapsto -z$. Choose a defining Hamiltonian  $F\colon W\to \R$  for $\Sigma$ and define the (non-equivariant) action functional { $\widetilde{\mathcal{A}}^{F,N} \colon \mathscr{P} \times (0,\infty) \times S^N \to \R$} by  
$$
\widetilde{\mathcal{A}}^{F,N}(c,\eta,z):=\ell(c(0)) - \ell(c(1)) + \int_0^1c^*\lambda - \eta\int_0^1 F(c(t))dt+h_N(z),
$$
where $\ell \colon L \to \R$ is a smooth function   satisfying $\lambda|_L = d \ell$.
Such a function exists since $L$ is an exact Lagrangian.
A critical point $(c,\eta,z) \in \crit(\widetilde{\mathcal{A}}^{F,N})$  satisfies
$$
\dot{c}(t)=\eta X_F(c(t)) \;\; \text{and} \;\;  c(t)\in \Sigma,\; \forall t \in [0,1],\;\; \text{and} \;\; z\in \crit(h_N).
$$
Since we have assumed that $\alpha$ is non-degenerate, $\widetilde{\mathcal{A}}^{F,N}$ is Morse.{
\begin{remark}
In the sprit of the ordinary Rabinowitz action functional \cite{CF09}, one can consider $\widetilde{\mathcal{A}}^{F,N}\colon \mathscr{P}\times \R \times S^N\to \R$ by replacing   $(0,\infty)$ with $\R$. In this case,  the functional is Morse--Bott due to the presence of critical points with $\eta=0$, and hence one has to take up Morse--Bott approach to define the homology.
Since only non-constant Reeb chords are considered in this paper, we will stick to our functional.
\end{remark}
}The map $\rho$ induces the involution on { $\mathscr{P}\times  (0,\infty) \times S^N$}
$$
\mathcal{R}(c(t),\eta,z)=\big(\rho(c(1-t)),\eta,-z\big).
$$
Since this $\mathcal{R}$-action is free and $\widetilde{\mathcal{A}}^{F,N}$ is $\mathcal{R}$-invariant,
we    obtain   the \emph{equivariant Rabinowitz action functional} $\mathcal{A}^{F,N}:\Omega\to \R$ on the quotient space {$\Omega:=(\mathscr{P}\times (0,\infty)\times S^N)/\mathcal{R}$}, 
 given by $\mathcal{A}^{F,N}([c,\eta,z])=\widetilde{\mathcal{A}}^{F,N}(c,\eta,z)$.
This functional is Morse as well.

%\subsubsection{The almost complex structures}
A  smooth family {${\bf J} = \{ J^{ \eta,z}_t \}_{(t, \eta, z) \in [0,1] \times (0,\infty) \times S^N}$} of $d\lambda$-compatible almost complex structures on $W$  
is called \emph{$\rho$-anti-invariant} if it satisfies
$$
J_t^{\eta, z}(x    )=-\rho _* \circ J_{1-t}^{\eta, -z}(\rho(x) ) \circ \rho_* 
$$
for all { $(t,\eta,z)\in [0,1]\times (0,\infty) \times S^N$} and $x\in W$.
We abbreviate by $\mathcal{J}_{\rho}^N$ the set of such smooth families of almost complex structures satisfying
\[
\sup_{t, \eta, z} \| J_t^{\eta, z} (\cdot) \|_{C^{\ell}} < \infty, \quad \ell \in \N.
\]
 Given ${\bf J} \in \mathcal{J}^N_{\rho}$, we obtain the $\mathcal{R}$-invariant $L^2$-metric on { $\mathscr{P}\times (0,\infty)  \times S^N$}, coupled with the standard inner product on { $(0,\infty)$} and the round metric $g_N$ on $S^N$,
\begin{equation}\label{eq: L2metric}
\langle (\hat{c}_1,\hat{\eta}_1,\hat{z}_1), (\hat{c}_2,\hat{\eta}_2,\hat{z}_2)\rangle = \int_0^1 d\lambda (\hat{c}_1,J_t(c,\eta)\hat{c}_2)dt + \hat{\eta}_1\cdot \hat{\eta}_2+g_N(\hat{z}_1,\hat{z}_2),	
\end{equation}
where { $(c,\eta,z)\in \mathscr{P}\times (0,\infty) \times S^N $ and $(\hat{c}_j, \hat{\eta}_j,\hat{z}_j)  \in T_{c}\mathscr{P} \oplus T_{\eta} (0,\infty)\oplus T_zS^N$} for $j=1,2$.
This induces an $L^2$-metric on $\Omega$.

%\subsubsection{The chain complex}\label{sec: chaincpx} 
Assume that the pair $(h_N,g_N)$ is Morse--Smale. The \emph{equivariant Rabinowitz Floer chain group} $\FC(F,h_N)$ is defined as the $\Z_2$-vector space consisting of formal sums  
$$
\zeta=\sum_{x\in \crit(\mathcal{A}^{F,N})} \zeta_x x,\quad \zeta_x\in \Z_2,
$$
meeting the following  finiteness condition: for every $\kappa\in \R$,
$$
\# \{x\in \crit(\mathcal{A}^{F,N})\mid \zeta_x\ne 0,\ \mathcal{A}^{F,N}(x)\ge \kappa \}<\infty.
$$
Fix $p_\pm=(c_\pm,\eta_\pm,z_\pm)\in \crit(\widetilde{\mathcal{A}}^{F,N})$. The space of (positive) gradient flow lines $\widehat{\mathcal{M}}(p_-,p_+) $ of $\widetilde{\mathcal{A}}^{F,N}$ from $p_-$ to $p_+$ with respect to the metric \eqref{eq: L2metric} consists of solutions {
 $$
 (c,\eta,z)\in C^\infty(\R\times [0,1],W)\times C^\infty(\R,(0,\infty))\times C^\infty(\R, S^N)
 $$
}of the equations
\begin{numcases}	{}
\nonumber \p_s c(s,t) + J_t^{z(s), \eta(s)}\big(c(s,t) \big)\Big(\p_t c(s,t) -\eta(s)X_F\big(c(s,t)\big)\Big) = 0,\label{eq: floereq} \\
\nonumber	\p_s \eta(s) + \int_0^1 F(c(s,t)) dt=0,\label{eq: lagmuleq} \\
	\p_s z(s)-\nabla_{g_N} h_N(z(s))=0.\label{eq: flowlineonSN}
\end{numcases}
with boundary condition $ 
c(s,0),c(s,1)\in L$ and asymptotic condition $$\displaystyle \lim_{s\to \pm\infty}(c(s),\eta(s),z(s))=(c_\pm,\eta_\pm,z_\pm).$$ The group $\R$ acts on $\widehat{\mathcal{M}}(p_-,p_+)$ by time-shift and the quotient space is denoted by $\mathcal{M}(p_-,p_+)$. Abbreviate $P_\pm := [p_{\pm}] \in \crit(\mathcal{A}^{F,N})$. The space of gradient flow lines from $P_-$ to $P_+$ is defined by
$$
\mathcal{M}(P_-,P_+) := \bigcup_{p_\pm \in P_\pm}\mathcal{M}(p_-,p_+).
$$
This space admits the free $\Z_2$-action induced by $\mathcal{R}$: 
$$
(c(s,t),\eta(s),z(s))\mapsto (\rho(c(s,1-t)),\eta(s),-z(s)).
$$
The quotient space
$$
\mathcal{M}_{\Z_2}(P_-,P_+)
$$
is called the \emph{moduli space of gradient flow lines from $P_-$ to $P_+$}.
%\begin{proposition}
%Let $P_\pm \in \crit(h_N)$. For generic $(g_N,{\bf J})$ the moduli space $\mathcal{M}_{\Z_2}(P_-,P_+)$ is a smooth manifold of dimension $|P_-|-|P_+|-1$.
%\end{proposition}
%The compactness result says that if $|P_-|-|P_+|=1$, then $\mathcal{M}_{\Z_2}(P_-,P_+)$ is compact. 
For a generic choice of ${\bf J}$ and $g_N$ this moduli space is a smooth manifold, see for example \cite{AM18}.
%\subsubsection{The equivariant Rabinowitz Floer homology}
For two critical points $P_{\pm} \in \crit(\mathcal{A}^{F,N})$, compactness results established in \cite{CF09, Me14} show that the zero-dimensional component $\mathcal{M}^0_{\Z_2}(P_-, P_+)$  of the moduli space $\mathcal{M}_{\Z_2}(P_-, P_+)$ is a finite set.
We set 
$\nu(P_-,P_+)\equiv \# \mathcal{M}_{\Z_2}^0(P_-,P_+)$ modulo two.
The \emph{equivariant Rabinowitz Floer differential}
$
\p^{{\bf J},g_N}\colon \FC(F,h_N)\to \FC(F,h_N)
$
 is then defined as the linear extension of
 $$
 \p^{{\bf J},g_N}(P )=\sum_{P'\in \crit(\mathcal{A}^{F,N})}\nu(P',P )P'.
$$
A standard argument shows that
$
\p^{J,g_N}\circ \p^{J,g_N}=0.
$ 
The resulting homology $\FH(F,h_N)$ does not depend on the choice of $(F,h_N,{\bf J},g_N)$ by the usual continuation argument.  
We then obtain canonical  homomorphisms  
$$
\iota_N: \FH(F,h_N)\to \FH(F,h_{N+1})
$$
using the following data:
\begin{itemize}
\item The inclusion $i_N \colon S^N \hookrightarrow S^{N+1}, z \mapsto (z,0)$ is $\Z_2$-equivariant and hence $\crit(\mathcal{A}^{F,N}) \subset \crit( \mathcal{A}^{F,N+1})$.
\item The Riemannian metric $g_{N+1}$ on $S^{N+1}$ satisfies $i_N^*g_{N+1}=g_N$.
\item Choose a Morse function $h_{N+1}$ on $S^{N+1}$ such that $i_N^*h_{N+1}=h_N$ and $(h_{N+1}, g_{N+1})$ is Morse--Smale.
\item The family ${ \bf J}_{N+1} = \{ J_t^{z,\eta} \}$ restricts to ${ \bf J}_N = \{ {J'}_t^{z,\eta}\}$, namely, $J_t^{z,\eta} = {J'}_t^{z,\eta}$ for $ z \in S^N$.
\end{itemize}
These homomorphisms commute with the continuation maps, so that we can take   the direct limit
$$
\RFH^{\Z_2}(\Sigma,W):=\varinjlim_{N} \FH(F,h_N).
$$
which is referred to as  the \emph{equivariant Lagrangian Rabinowitz Floer homology} of the hypersurface $\Sigma$ in $W$.

\subsubsection{The case of cotangent bundles} \label{rmk:cotanegung}
One can define a grading on the equivariant Lagrangian Rabinowitz Floer homology groups, but it is in general not globally well-defined. 
 For cotangent bundles, however, there exists a so-called vertical preserving symplectic trivialization, see \cite{APS08,AS06}, 
 which enables us to define a well-defined global grading.

 We  recall the setup in Section \ref{sec:main}. 
 Let   $N$ be a  closed and  connected $n$-dimensional smooth manifold equipped with an involution $f$, whose fixed point set  $Q = \mathrm{Fix}(f) \subset N$ is non-empty and connected.
 The involution $f$ induces the exact anti-symplectic involution $\rho$ on $(T^*N, \lambda)$, and its fixed point set is given by the conormal bundle $N^*Q$ of $Q$ in $T^*N$, which  is a connected exact Lagrangian submanifold.   
Let $\Sigma \subset T^*N$ be a $\rho$-invariant fiberwise starshaped hypersurface   which is transverse to $N^*Q $. Abbreviate  $\mathcal{L}=\Sigma \cap N^*Q$, which is assumed to be non-empty.

 For a non-degenerate Reeb chord $(c,\eta)$ on $\Sigma$ with endpoints in $\mathcal{L}$,  we define the   index $\mu(c,\eta)$ as {
\begin{equation}\label{eq: maslovindex}
 \mu(c, \eta) := \mu_{\mathrm{RS}}(c,\eta)-\frac{n-1}{2},
\end{equation}
}where $\mu_{\mathrm{RS}}$ denotes the  Robbin--Salamon index, see \cite{RSindex}. 
   The index of a critical point   $P = [(c,\eta, z)]$ of $\mathcal{A}^{F,N}$ is then defined as
 \[
 | P | := \mu(c,\eta) + \mathrm{ind}(z; h_N)\in \Z,
 \]
where $\mathrm{ind}(z;h_N)$ denotes the Morse index of $z\in \crit(h_N)$. Arguing as in \cite[Proposition 1.3]{AM18}, or  \cite[Appendix A]{CF09}, 
we find that
if $P_{\pm} = [(c_{\pm}, \eta_{\pm}, z_{\pm})] \in \crit(\mathcal{A}^{F,N})$, then 
\begin{align*}
\dim \mathcal{M}_{\Z_2}   (P_-,P_+) &= \mu(c_+, \eta_+) - \mu(c_-, \eta_-) + \mathrm{ind} (z_+;h_N) - \mathrm{ind} (z_-;h_N) -1 \\
&= \lvert P_+ \rvert - \lvert P_- \rvert -1.
\end{align*}
 Therefore, if $|P_+| - |P_-| =1$, then $\mathcal{M}_{\Z_2}(P_-, P_+)$  is a compact zero-dimensional manifold and hence  a finite set.
 We define the boundary operator by counting the elements in this space module two and obtain a graded Floer homology group $\FH_*(F,h_N)$.
 In this case, the canonical homomorphisms  $\iota_N: \FH_*(F,h_N)\to \FH_*(F,h_{N+1})$ respects the grading. 
 The rest of the construction of the equivariant Lagrangian Rabinowitz Floer homology groups is word-for-word identical to the general case.

\begin{remark}
Our grading differs from  the one in \cite{Me14} for   Lagrangian Rabinowitz Floer homology of cotangent bundles  by a global shift.

\end{remark}

\subsection{Local non-equivariant Lagrangian Rabinowitz Floer homology}\label{sec: localnonequi}
In this section, we work locally near a given   Reeb chord.  
Let $(W, \lambda)$ be an  exact symplectic manifold and let $\Sigma$ be a    hypersurface   in $W$ that is of contact type with respect to a contact form $\alpha := \lambda|_{\Sigma}$.
We \textit{do not} assume that $\alpha$ is non-degenerate.
Suppose that $L$ is a non-empty connected exact Lagrangian submanifold in $W$ such that it intersects $\Sigma$ transversally and the intersection $\mathcal{L}=L \cap \Sigma$ is a non-empty Legendrian submanifold in $\Sigma$.
  
Consider the (non-equivariant) Rabinowitz action functional { $\mathcal{A}^F\colon \mathscr{P}\times (0,\infty)\to \R$} given by
$$ 
\mathcal{A}^F(c,\eta):=\ell(c(0)) - \ell( c(1)) + \int_0^1 c^* \lambda - \eta \int_0^1 F(c(t))dt.
$$ 
Under generic perturbations of $F$, this functional is Morse.

For a subset $\mathcal{F}\subset \crit(\mathcal{A}^F)$ we write
$$
\im(\mathcal{F}):=\bigcup_{(c,\eta)\in \mathcal{F}} \im (c) \times \{\eta\}\subset W\times (0,\infty).
$$
A (possibly disconnected) subset $\mathcal{F}$ of $\crit(\mathcal{A}^F)$ is called \emph{isolated} if there is an open neighborhood $U$ of $\im(\mathcal{F})$ in $W\times (0,\infty)$ such that the closure $\overline{U}$ is compact and $\overline{U}\cap \im(\crit(\mathcal{A}^F))=\im(\mathcal{F})$. We call $U$ an \emph{isolating neighborhood} of $ \mathcal{F} $. Without loss of generality we may assume that
$$
U=U_1\times U_2
$$
for some open sets $U_1\subset W$ and $U_2\subset (0,\infty)$. Now we define the \emph{local non-equivariant Lagrangian Rabinowitz Floer homology} $\RFH^{\loc}(\mathcal{F})$ of $\mathcal{F}$.   Intuitively, this is the  Floer homology associated to the local Rabinowitz action functional
$$
\mathcal{A}^{F}_{U} \colon \mathscr{P}_{U_1}\times U_2\to \R,\quad \mathcal{A}^F_{U_1}(c,\eta)=\mathcal{A}^F(c,\eta),
$$
where $\mathscr{P}_{U_1}=\{c\in \mathscr{P}\mid \im(c)\subset U_1\}$.
In our construction below, we follow  \cite{CFHW,GinzGurel, Mclean}, where more standard local Floer homologies are constructed.

We begin with the following crucial lemma, which can be proved by a slight modification of the proof of  \cite[Lemma 2.1]{CFHW} by applying  the  compactness result \cite[Theorem 3.1]{CF09}.

\begin{lemma}\label{lem: localconv} Let $\mathcal{F}$ be an isolated subset of $\crit(\mathcal{A}^F)$ and $U=U_1 \times U_2 $ an isolating neighborhood of $ \mathcal{F} $. Suppose that all  Reeb chords   in  $\mathcal{F}$ have the common length $\eta_{\mathcal{F}}$.
 	 Let $F_\nu$ and ${\bf J}_\nu$ be   sequences of Hamiltonians converging to $F$  and  of $d\lambda$-compatible almost complex structures converging to a  $d\lambda$-compatible almost complex structure ${\bf J}$ in the $C^\infty$-topology, respectively.
 Then for any neighborhood $V=V_1\times V_2$ of $\im(\mathcal{F})$  with $\overline{V} \subset U$, there exists $\nu_0\gg 1$ such that for $\nu\ge \nu_0$ we have
\begin{itemize}
	\item $($Critical points$)$ If $(c,\eta)\in \crit(\mathcal{A}^{F_\nu})$ such that $\im(c)\times\{\eta\}\subset U$, then $\im(c)\times \{\eta\}\subset V$.
	\item $($Gradient flow lines$)$ %All $(\tilde{c},\tilde{\eta})\in \mathcal{M}((c^-,\eta^-),(c^+,\eta^+);\mathbf{J},F,\mathcal{U})$ are contained in $\mathcal{V}$.
If
$
(\tilde{c} \colon \R\times [0,1]\to U_1,\ \tilde{\eta} \colon \R \to U_2)
$ is a gradient flow line of $\mathcal{A}^{F_\nu}$ $($with respect to the $L^2$-metric induced by  $({\bf J}_\nu, g_N)$$)$ connecting critical points $(c^\pm,\eta^\pm)\in \crit(\mathcal{A}^{F_\nu})$ such that $\im(c^\pm)\times\{\eta^\pm\}\subset U$, then $\im(\tilde{c},\tilde{\eta})\subset V$.
\end{itemize}
\end{lemma}
Let $\mathcal{F}\subset \crit(\mathcal{A}^F)$ be an isolated subset such that all   Reeb chords in $\mathcal{F}$ have the common length $\eta_\mathcal{F}$. 
Choose an isolating neighborhood $U=U_1 \times U_2$  of $ \mathcal{F} $ and a smaller neighborhood $V=V_1 \times V_2 \subset U$ of $\im(\mathcal{F})$ such that $\overline{V}\subset U$. Now take a Hamiltonian $\tilde{F}\colon W\to \R$ with the following properties.
\begin{enumerate}
	\item On $U_1$, the Hamiltonian $\tilde{F}$ is $C^2$-close to $F$. %(so that Lemma \ref{eq: localnotin2} holds { and it satisfies that if $(\im(c),\eta)\in \mathcal{U}$ for some $(c,\eta)\in \crit(\mathcal{A}^{\tilde{F}})$, then $(c(0),\eta)\in U$.}),
	\item $\tilde{\Sigma}:=\tilde{F}^{-1}(0)\cap U_1$ is a contact-type regular hypersurface,
	\item $L$ intersects  $\tilde{\Sigma}$ transversally,
	\item $\lambda|_{\tilde{\Sigma}}$ is non-degenerate.
\end{enumerate}
Here we made use of \cite[Theorem B.1]{CF09} and the fact that the contact-type and transversality conditions are open.

The \emph{local Rabinowitz Floer chain group} $\FC^{\loc}(U;\tilde{F})$ is defined as the $\Z_2$-vector space generated by the critical points of $\crit(\mathcal{A}^{\tilde{F}}_U)$, i.e.,
$$
\FC^{\loc}(U;\tilde{F})=\bigoplus_{x\in \crit(\mathcal{A}^{\tilde{F}})} \Z_2\langle x \rangle.
$$
Since the Lagrange multipliers of $\crit(\mathcal{A}^{\tilde{F}}_U)$ are uniformly bounded, we do not use the Novikov-type coefficients as in the previous section. For a generic ${\bf J}$, the \emph{local Rabinowitz Floer differential}
$$
\p^{\loc,{\bf J}}\colon \mathrm{FC}^{\loc}(U;\tilde{F})\to \mathrm{FC}^{\loc}(U;\tilde{F})
$$ 
is defined by counting rigid elements of the moduli spaces of gradient flow lines that are  contained in $U$. By Lemma \ref{lem: localconv}, all critical points of $\crit(\mathcal{A}^{\tilde{F}}_{U})$ and $(\tilde{F},{\bf J})$-gradient flow lines have image   in $V$, and hence are bounded uniformly away   from the boundary $\p U$ (possibly with corners). Hence, a standard argument shows that $\p^{\loc,{\bf J}}$ is well-defined and satisfies $\p^{\loc,{\bf J}}\circ \p^{\loc,{\bf J}}=0$. The resulting homology 
$$
\RFH^{\loc}(\mathcal{F}):=H\Big(\FC(U;\tilde{F}),\p^{\loc, {\bf J}}\Big)
$$
is called the \emph{local Rabinowitz Floer homology} of $\mathcal{F}$. By a usual continuation argument, which requires a parametrized version of Lemma \ref{lem: localconv}, the homology does not depend on the choice of $(U,V,\tilde{F},{\bf J})$.

\subsection{Local equivariant Lagrangian Rabinowitz Floer homology} \label{sec:local} As in the previous section,  $(W, \lambda)$ is an  exact symplectic manifold and   $\Sigma$ is a    hypersurface   in $W$ that is of contact type with respect to a contact form $\alpha := \lambda|_{\Sigma}$.
We again \textit{do not} assume that $\alpha$ is non-degenerate. Assume that $(W, \lambda)$ admits an exact anti-symplectic involution $\rho$.
The restriction $\rho|_{\Sigma}$ is an anti-contact involution on $(\Sigma, \alpha)$.
Assume further that the fixed point set $L= \mathrm{Fix}(\rho)$ is non-empty and connected, so that  $L$ is an exact Lagrangian submanifold in $(W,\lambda)$.
The intersection $\mathcal{L} = L \cap \Sigma$ is then a Legendrian submanifold in $(\Sigma, \alpha)$, and we also assume that $\mathcal{L}$ is non-empty.
Suppose that $F\colon W \to \R$ is a  defining Hamiltonian for $\Sigma$ which is $\rho$-invariant. 
The functional $\widetilde{\mathcal{A}}^{F,N} \colon \mathscr{P} \times (0,\infty) \times S^N \to \R$ is defined in the same way as in Section \ref{sec:eLRH}. %Any subset $\widetilde{\mathcal{F}}$ of $\crit(\widetilde{\mathcal{A}}^{F,N}_+)$ is of the form $\widetilde{\mathcal{F}}=\mathcal{F}\times S^N$ for some $\mathcal{F}\subset \crit(\mathcal{A}^F)$. 
We say that  $\mathcal{F}\subset \crit(\mathcal{A}^F)$ is \emph{$\Z_2$-invariant} if it  is invariant under the involution $\mathcal{R} (c,\eta)=(\rho(c(1-t)),\eta)$ on $\mathscr{P}\times (0,\infty)$. Such an $\mathcal{F}$ gives rise to an $\mathcal{R}$-invariant subset $\widetilde{\mathcal{F}}:=\mathcal{F}\times \crit(h_N)\subset \crit(\widetilde{\mathcal{A}}^{F,N})$.

Let $\mathcal{F}$ be an isolated subset of $\crit(\mathcal{A}^F)$ that is $\Z_2$-invariant. An isolating neighborhood $U=U_1\times U_2
$ of $\mathcal{F}$ in $W\times (0,\infty)$ is called \emph{$\Z_2$-invariant} if $U_1$ is $\rho$-invariant. Any $\Z_2$-invariant isolating neighborhood $U$ of $\mathcal{F}$ gives rise to an $\mathcal{R}$-invariant isolating neighborhood $\widetilde{U}:=U\times S^N$ for $\widetilde{\mathcal{F}}=\mathcal{F}\times \crit(h_N)$. The following lemma is immediate from Lemma \ref{lem: localconv}.

\begin{lemma}\label{lem: localconveq} With the notation from   Lemma \ref{lem: localconv},   for any $\Z_2$-invariant neighborhood $V=V_1\times V_2$ of $\mathcal{F}$ whose closure is contained in $U$, there exists $\nu_0\gg 1$ such that for $\nu\ge \nu_0$ we have
\begin{itemize}
	\item $($Critical points$)$ If $(c,\eta,z)\in \crit(\widetilde{\mathcal{A}}^{F,N})$ such that $\im(c)\times \{\eta\}\times \{z\}\subset \widetilde{U}$, then $\im(c)\times \{\eta\}\times \{z\} \in \widetilde{V}$.
	\item $($Gradient flow lines$)$
If
$
(\tilde{c} \colon \R\times [0,1]\to U_1,\ \tilde{\eta} \colon \R \to U_2,\  \tilde{z}\colon \R \to S^N)
$ is a gradient flow line of $\widetilde{\mathcal{A}}^{F_\nu,N}$ $($with respect to the metric induced by $({\bf J}_\nu,g_N))$ connecting critical points $(c^\pm,\eta^\pm,z^\pm)\in \crit(\widetilde{\mathcal{A}}^{F_\nu,N})$ such that $\im(c^\pm)\times \{\eta^\pm\}\times \{z^\pm\}\in \widetilde{U}$, then $\im(\tilde{c},\tilde{\eta},\tilde{z})\subset \widetilde{V}$.
\end{itemize}
\end{lemma}
Let $\widetilde{\mathcal{F}} = \mathcal{F}\times \crit(h_N)$, $\widetilde{U}$, and $\widetilde{V}$ be as in the previous lemma. Considering the local action functional
$$
\widetilde{\mathcal{A}}^{F,N}_{\widetilde{U}}\colon \mathscr{P}_{U_1}\times U_2\times S^N \to \R,
$$ 
we shall define $\RFH^{\loc, \Z_2}(\widetilde{\mathcal{F}})$ for   $\widetilde{\mathcal{F}}$.   The construction is exactly   as in Section \ref{sec: localnonequi}: We take a $\rho$-invariant Hamiltonian $\tilde{F}\colon W \to \R$, which is  a $C^2$-small perturbation of $F$ on $U_1$, satisfying   properties (1)--(4) given in Section \ref{sec: localnonequi} and assume that the pair $(h_N,g_N)$ is Morse--Smale.

The \emph{equivariant local Rabinowitz Floer chain group} $\FC^{\loc,\Z_2}(\widetilde{U};\tilde{F},h_N)$ is defined as the $\Z_2$-vector space generated by  the critical points of $\mathcal{A}^{\tilde{F},N}_{\tilde{U}}$.  Choose a generic ${\bf J}$ which is  $\rho$-anti-invariant. The \emph{local equivariant   Rabinowitz Floer differential}
$$
\p^{\loc, \Z_2, {\bf J},g}\colon \FC^{\loc,\Z_2} (\widetilde{U};\tilde{F},h_N)\to \FC^{\loc,\Z_2} (\widetilde{U};\tilde{F},h_N)
$$ 
is defined by counting rigid elements in the moduli spaces of those gradient flow lines that are contained in $\tilde{U}$. By Lemma \ref{lem: localconveq}, the differential $\p^{\loc, \Z_2,{\bf J},g_N}$ is well-defined and a standard argument shows that $ \p^{\loc, \Z_2,{\bf J},g_N} \circ \p^{\loc, \Z_2,{\bf J},g_N}=0$. 
As in Section \ref{sec:eLRH} we have canonical homomorphisms 
\[
\iota_N \colon  \FH^{\loc, \Z_2}(\widetilde{U};\tilde{F},h_N) \to  \FH^{\loc, \Z_2}(\widetilde{U};\tilde{F},h_{N+1}) 
\]
and the associated direct limit   
$$
\FH^{\loc,\Z_2}(\widetilde{U};\tilde{F}) :=\varinjlim_{N} \FH^{\loc, \Z_2}(\widetilde{U};\tilde{F},h_N)
$$
is independent of the choice of $(U,V,\tilde{F},h_N,{\bf J},g_N)$ by a continuation argument. The homology $\RFH^{\loc, \Z_2} (\widetilde{\mathcal{F}}) := \FH^{\loc,\Z_2}(\widetilde{U};\tilde{F})$ is referred to as the \emph{local equivariant  Lagrangian Rabinowitz Floer homology} of $\widetilde{\mathcal{F}}$.
In particular, we have the following invariance property. See \cite[p.\ 275]{CF09} for details.
\begin{proposition}\label{prop: invariance}
	We have a canonical isomorphism
	$$
	\FH^{\loc, \Z_2} (\widetilde{U};\tilde{F}_0)\cong 	\FH^{\loc, \Z_2} (\widetilde{U};\tilde{F}_1)
	$$
for any two $\tilde{F}_0,\tilde{F}_1$ satisfying  properties of (1)--(4) given in Section \ref{sec: localnonequi}.
\end{proposition}
\begin{remark}
	In the case that $\FH_*^{\loc, \Z_2}$ is equipped with  a $\Z$-grading, the above isomorphism respects the grading.
\end{remark} 
From now on, we restrict our setup to the case $W=T^*N$ equipped with the  Liouville 1-form $\lambda$, where $N$ is an $n$-dimensional smooth manifold. By the discussion  in Section \ref{rmk:cotanegung}, the local eLRFH admits a $\Z$-grading.   The following proposition tells us that if $\mathcal{F}$ consists of exactly one $\Z_2$-pair of Reeb chords, then $\RFH_*^{\loc,\Z_2}(\widetilde{\mathcal{F}})$ has only one generator, whose index is given by the  index of the Reeb chord.

\begin{proposition}\label{prop: eRFHcomputation}
Suppose that the  $\Z_2$-invariant subset 
	$ {\mathcal{F}}=\{(c,\eta ),(\bar{c},\eta )\}$ of $\crit(\mathcal{A}^{F})$ is isolated and $(c, \eta)$ is non-degenerate.  In particular, $\eta \neq 0$. Let $U=U_1\times U_2$ be an isolating neighborhood of $\mathcal{F}$.  Then we have
	$$
	\FH_*^{\loc, \Z_2}(\widetilde{U};F) = \begin{cases}
		\Z_2 & \text{if $*=\mu(c,\eta) $},\\
		0 & \text{otherwise},
	\end{cases}
	$$
	where $\mu(c,\eta)=\mu_{
\RS}(c,\eta)-\frac{n-1}{2}$.
	In particular, $\RFH_*^{\loc, \Z_2}(\widetilde{\mathcal{F}})=\FH_*^{\loc, \Z_2}(\widetilde{U};F)$.
\end{proposition}
\begin{proof}
Fix a sequence of Morse functions and metrics
$
\{(h_N\colon S^N\to \R,g_N)\}_{N\in \N}
$
such that for all $N\in \N$
\begin{itemize}
	\item $h_N$ and $g_N$ are invariant under the antipodal map;
	\item the pair $(f_N,g_N)$ is Morse--Smale; and
	\item $i_N^*f_{N+1}=f_N$ and $i_N^*g_{N+1}=g_N$ where $i_N\colon S^N\hookrightarrow S^{N+1}$ denotes the equivariant inclusion.
\end{itemize}
Since the chain complex $(\FC^{\loc, \Z_2}_*(\tilde{U};\tilde{F},h_N),\p^{\loc,\Z_2}_*)$ is canonically identified with the Morse chain complex associated to the pair $(h_N,g_N)$ up to degree shift by $\mu(c,\eta)$, the homology $\FH_*^{\loc,\Z_2}(\widetilde{U};F,h_N)$ has exactly two generators with indices $\mu(c,\eta)$ and $\mu(c,\eta)+N$. Taking the direct limit with respect to $N\to \infty$, we obtain the desired result.
\end{proof}
Consider a family of Hamiltonians  $F_\tau $, $\tau\in \R$. We assume that $\Sigma_\tau:=F^{-1}(\tau)$ is a regular hypersurface which is of contact-type for all $\tau\in (-\epsilon,\epsilon)$ for some $\epsilon>0$.  
We conclude this section with providing the following definition.

\begin{definition}\label{def: unifomlyisolated}
Let $\{\mathcal{F}_\tau\}_{\tau\in (-\epsilon, \epsilon)}$ be a smooth family of subsets $\mathcal{F}_\tau$ of $\crit(\mathcal{A}^{F_\tau})$ and $I\subset (-\epsilon,\epsilon)$ an interval. A subfamily $\{\mathcal{F}_\tau\}_{\tau\in I}$ is called  \emph{uniformly isolated} if 
\begin{itemize}
	\item $\mathcal{F}_\tau$ is isolated for each  $\tau\in I$ in the sense of Section \ref{sec: localnonequi}; and
	\item There exists an open set $U = U_1 \times U_2\subset W\times (0,\infty)$ such that $U$ is an isolating neighborhood  of $\mathcal{F}_\tau$ for all $\tau\in I$.
\end{itemize}

\end{definition}

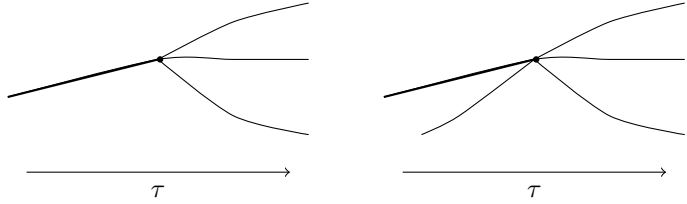
\begin{figure}[h]
\begin{tikzpicture}[scale=0.5]
%\draw[step=0.5,black,thin] (-5,-5) grid (5,5);
\draw plot [smooth] coordinates {(-3,0) (1,1) (3,1) (5,1) };
\draw [thick] plot [smooth] coordinates {(-3,0) (1,1)};
\draw plot [smooth] coordinates {(1,1) (3,2) (5,2.5) };
\draw plot [smooth] coordinates {(1,1) (3,-0.5) (5,-1) };
\draw [->] (-2.5,-2)--(4.5,-2);
\draw node at (1,-2.5){$\tau$};
%\draw node at (0, -2){uniformly isolated};
\draw [fill] (1.05,1) circle (0.07cm);
\begin{scope}[xshift=10cm]
\draw [fill] (1.05,1) circle (0.07cm);
\draw plot [smooth] coordinates {(-3,0) (1,1) (3,1) (5,1) };
\draw [thick] plot [smooth] coordinates {(-3,0) (1,1)};
\draw plot [smooth] coordinates {(1,1) (3,2) (5,2.5) };
\draw plot [smooth] coordinates {(1,1) (3,-0.5) (5,-1) };	
\draw plot [smooth] coordinates {(1,1) (-1,-0.5) (-2,-1) };	
\draw [->] (-2.5,-2)--(4.5,-2);
\draw node at (1,-2.5){$\tau$};
\end{scope}

\end{tikzpicture}
  \caption{The subfamily indicated by a bold line on the left-hand side is uniformly isolated, while the right one is not.}
\label{Fig:unifiso}
\end{figure}
 
 \section{Proofs of Theorems \ref{thm: intromainthm}, \ref{thm: intromainthm1}, and \ref{thm: intromainthm2} and of Corollary \ref{thm: intromain}}\label{sec:proofs}

 We are in position to  prove  the theorems and the  corollary stated in Section \ref{sec:bifurc}.
\begin{proof}[Proof of Theorem \ref{thm: intromainthm}  ]   Since $ \{(x_{c_\tau},2\eta_\tau) \}_{\tau\in (-\epsilon,0]}$ is uniformly isolated,  we find an open set $U=U_1 \times U_2 \subset T^*N \times (0,\infty)$ which is an isolating neighborhood of $\{ (c_{\tau}, \eta_{\tau}), (\overline{c}_{\tau}, \eta_{\tau})\}$ for all $\tau \in (-\epsilon, 0]$. This in particular implies that for any $\tau\in (-\epsilon,0]$, we have
$$
\crit(\mathcal{A}^{F_\tau}_U)=\{(c_\tau,\eta_\tau),(\bar{c}_\tau,\eta_\tau) \}.
$$
By Proposition \ref{prop: eRFHcomputation},  for all  $\tau\in (-\epsilon,0)$ close   enough to $0$ we obtain that
\begin{equation}\label{eq: mainresultcomputation}
	 \FH_*^{\loc, \Z_2}(\widetilde{U};F_\tau)= \begin{cases}
		\Z_2 & \text{if $*=\mu(c_{\tau},\eta_{\tau})$},\\
		0 & \text{otherwise}.
	\end{cases}
\end{equation}
We claim that there exists $\epsilon_0>0$ with the following property: for all $0<\tau<\epsilon_0$, there exists a   $\rho$-symmetric periodic orbit $(x_{d_\tau}, 2\sigma_{\tau})$  different from  $(x_{c_\tau}, 2 \eta_{\tau})$. Assume to the contrary that there exists a sequence $\tau_\nu\to 0^+$ such that for all $\nu\in \N$ we have 
$$
\crit(\mathcal{A}^{F_{\tau_\nu}}_U)=\{(c_{\tau_\nu},\eta_{\tau_\nu}),(\bar{c}_{\tau_\nu},\eta_{\tau_\nu})\}.
$$
If $\nu$ is large enough and $0<\delta<\epsilon$ is  so small that $F_{\tau_\nu}$ and $F_{-\delta}$ satisfy the hypothesis of Lemma \ref{lem: localconveq}, then by Proposition \ref{prop: invariance} we obtain a grading-preserving isomorphism
$$
\FH_*^{\loc, \Z_2}(\widetilde{U};F_{-\delta})\cong \FH_*^{\loc, \Z_2}(\widetilde{U};F_{\tau_\nu}).
$$
On the other hand, in view of  Proposition \ref{prop: eRFHcomputation}, for $\nu$ large enough we find that
\begin{equation}\label{eq: result2}
	 \FH_*^{\loc, \Z_2}(\widetilde{U};F_{\tau_\nu})= \begin{cases}
		\Z_2 & \text{if $*=\mu(c_{{\tau_\nu}},\eta_{{\tau_\nu}})$},\\
		0 & \text{otherwise}.
	\end{cases}
\end{equation}
Due to the hypothesis that $\mu(c_{-\delta},\eta_{-\delta})\ne \mu(c_{\tau_\nu},\eta_{\tau_\nu})$, identities \eqref{eq: mainresultcomputation} and \eqref{eq: result2} yield   a contradiction and  this proves the claim.
It follows that there exists $\delta>0$ small enough such that we obtain an additional    pair  of  (possibly degenerate) Reeb chords $\{(d_\tau,\sigma_\tau), (\bar{d}_{\tau}, \sigma_{\tau})\}_{\tau \in (0, \delta)}$   whose associated symmetric periodic orbits $(x_{d_\tau}, 2\sigma_{\tau})$ differ from $(x_{c_\tau}, 2 \eta_{\tau})$. 
Since  one can choose $U$ as small as one likes, it follows that the symmetric periodic orbits $(x_{d_\tau}, 2\sigma_{\tau})$ tend to  $(x_{c_0}, 2 \eta_{0})$ as $\tau \to 0^+$.
This proves the first assertion of the theorem.

To prove the second assertion, assume that there exists a family of pairs of simple Reeb chords  $\{ ( {\mathfrak{c}}_\tau,  \ell_\tau), (\bar{\mathfrak{c}}_{\tau}, \ell_{\tau})\}_{\tau \in (-\epsilon, \epsilon)}$ such that   $ (c_\tau,\eta_\tau) = (\mathfrak{c}_\tau^n , n \ell _\tau)$ for some $n$ and such that     $ (\mathfrak{c}^m_\tau , m \ell _\tau) $ is non-degenerate for every $\tau$ and   for every $m <n$.
Suppose that for each $\tau\in(0, \delta)$ we have $(d_{\tau}, \sigma_{\tau}) = ( \mathfrak{d}_{\tau}^p, p \mathfrak{s}_{\tau}  )$ for some $p>0$, where $(d_{\tau}, \sigma_{\tau})$ is the   Reeb chord obtained in the first assertion and $(\mathfrak{d}_{\tau},  \mathfrak{s}_{\tau} )$ is some simple Reeb chord.

This, however, implies that the family $\{ (x_{\mathfrak{d}_{\tau}},  2\mathfrak{s}_{\tau} )    \}_{\tau \in (0, \delta)}$ is born out of the family $\{(x_{\mathfrak{c}_{\tau }}^{n/p},(n/p)2\eta_{\tau }) \}_{\tau \in (-\epsilon, \epsilon)}$, which contradicts the condition on the minimality of $n$. This proves the second assertion.

To conclude the proof, we now assume that $\Sigma_{\tau}$ is compact for every $\tau \in (0, \delta)$ and that all Reeb chords with image  in $\Sigma_{\tau} \cap U_1$ and with lengths in $U_2$ for $\tau \in (0, \delta)$ are non-degenerate.  
We claim that  there exists $\delta_0 \in (0, \delta]$ such that  for all $\tau \in [0, \delta_0)$ we have $\im ( \crit (\mathcal{A}^{F_{\tau}})) \cap \p \overline{U} = \emptyset$.
 Recall that the closure $\overline{U}$ of $U$ is compact. 
Arguing by contradiction we assume that there is no $\delta_0 \in (0, \delta]$ having the property of the claim. 
Then we  find a sequence $\tau_{\nu} \in (0, \delta)$ such that $\tau_{\nu} \to 0^+$ as $\nu \to \infty$ and 
a sequence $(c_{\tau_{\nu}}, \eta_{\tau_{\nu}}) \in \crit ( \mathcal{A}^{F_{\tau}})$  satisfying $\left( \im (c_{\tau_{\nu}})  \times \{ \eta_{\tau_{\nu}} \}    \right) \cap \p \overline{U} \neq \emptyset$.
It follows that there exists a sequence $t_{\nu} \in [0,1]$ such that $\left( c_{\tau_{\nu}}( t_{\nu}) , \eta_{\tau_{\nu}} \right) \in \p \overline{U}$. 
Thanks to the compactness of $\overline{U}$ and  of the hypersurfaces $\Sigma_{\tau}$,
we are able to apply the Arzel\`a-Ascoli theorem to find $t_0 \in [0,1]$ and $(c_0', \eta_0') \in \crit( \mathcal{A}^{F_0})$ different from $(c_0, \eta_0)$ such that    $\left( c_{\tau_{\nu}}( t_{\nu}) , \eta_{\tau_{\nu}} \right) \to \left( c_{0}'( t_0) , \eta_0' \right) \in \p \overline{U}$   as  $\nu \to \infty.$
This contradicts to the property of $U$ that $\overline{U} \cap \im( \crit (\mathcal{A}^{F_0})) = \im ( c_0) \times \{ \eta_0 \}$ and proves the claim.

We fix $\tau_0 >0 $ small enough and let $(d_{\tau_0}, \sigma_{\tau_0})$ be a Reeb chord on $\Sigma_{\tau}$ obtained by the first assertion.
Since it is non-degenerate, making use of the implicit function theorem, see for example, \cite[Proposition 2, page 110]{HZbook} or \cite[Proposition B.1]{AF09negative},
we find $\tau_{-\infty} \in [0, \tau_0)$ such that $(d_{\tau_0}, \sigma_{\tau_0})$ extends to a maximal one-parameter family  $\{ (d_{\tau}, \sigma_{\tau})\}_{\tau \in ( \tau_{-\infty}, \tau_0]   }$ such that
for each $\tau  \in ( \tau_{-\infty}, \tau_0] $ the pair $(d_{\tau}, \sigma_{\tau})$ is a non-degenerate Reeb chord on $\Sigma_{\tau}$.
In view of the previous claim, each $\im (d_{\tau})  $ is confined to $\Sigma_{\tau} \cap U_1$.
It follows from the hypothesis that   all these Reeb chords are non-degenerate and $\tau_{-\infty} =0$.

Consider the $\omega$-limit set $\Omega$ of the family $\{ (d_{\tau}, \sigma_{\tau})\}_{\tau \in (  0, \tau_0]   }$. By definition, $\Omega$ consists of Reeb chords $(v, \zeta)$ on $\Sigma_0$ such that
there exists a sequence $\tau_{\nu} \in (0, \tau_0]$ having the properties that $\tau_{\nu} \to 0$ and $ ( d_{\tau}, \sigma_{\tau}) \to (v, \zeta)$ as $\nu \to \infty$.
\cite[Theorem A]{BFvK19} tells us that $\Omega$ is nonempty, compact, and connected. 
However, the degenerate chord $(c_0, \eta_0)$  is the only Reeb chord with image in $\Sigma_0 \cap U_1$ and with lengths in $U_2$, and hence we have $\Omega = \{ (c_0, \eta_0)\}$.

We have found a one-parameter family $\{ (x_{d_{\tau}}, 2 \sigma_{\tau})\}_{\tau \in (0, \delta_0)}$ of $\rho$-symmetric periodic orbits, with $\delta_0>0$ small enough, born out of the family $\{(x_{c_{\tau}}, 2\eta_{\tau})\}_{\tau \in (-\epsilon, \epsilon)}$. 
Suppose that $\mu (c_{\tau}, \eta_{\tau}) = k$ for $\tau <0$ and $\mu(c_{\tau}, \eta_{\tau}) = k+ p$ for some $p \in [ -\dim N +1, \dim N-1] \setminus \{ 0 \}$ for $\tau>0$, see Remark \ref{rmk:index}.
As above, if $\kappa>0$ is small enough, then we have
\begin{equation}\label{eq:RKFfhormula}
\FH_*^{\loc, \Z_2}(\widetilde{U};F_{-\kappa})\cong \FH_*^{\loc, \Z_2}(\widetilde{U};F_{\kappa}) = \begin{cases}
		\Z_2 & \text{if $*=k$},\\
		0 & \text{otherwise}.
	\end{cases}
\end{equation}
This tells us that for $\tau>0$ sufficiently small, 
we need at least two Reeb chords: one $(d_{\tau}, \sigma_{\tau})$ with index $k$ is the generator of $\FH_*^{\loc, \Z_2}(\widetilde{U};F_{\kappa})$ and the other $(d_{\tau}', \eta_{\tau}')$ satisfying one of the following
\begin{itemize}
\item It has index  $k+p-1$ and is killed by $(c_{\tau}, \eta_{\tau})$ in the construction of $ \FH_*^{\loc, \Z_2}(\widetilde{U};F_{\kappa}) $.
\item Its index is   $k+p+1$ and it kills $(c_{\tau}, \eta_{\tau})$  in the construction of $ \FH_*^{\loc, \Z_2}(\widetilde{U};F_{\kappa}) $.
\end{itemize}
 This proves the third assertion and finishes the proof of the theorem. 
\end{proof}

 \begin{proof}[Proof of Corollary \ref{thm: intromain}]
 
Applying Theorem \ref{thm: intromainthm} to the family $\{  (y_{\tau}, 4\eta_{\tau})\} _{\tau \in (-\epsilon, \epsilon)}$ by regarding it as a family of $\rho_2$-symmetric periodic orbits, we find  new $\rho_2$-symmetric periodic orbits $(x_{d_{\tau}}, 2\sigma_{\tau})$ for all $ \tau \in (0, \delta)$ and some small $\delta >0$. 
We observe that each $(x_{d_{\tau}}, 2\sigma_{\tau})$ is not $\rho_1$-symmetric. 
Indeed, assume that it is $\rho_1$-symmetric. Since   the Reeb chord $c_{1,{\tau}}$ is non-degenerate for all $\tau \in (-\epsilon, \epsilon)$, no bifurcation of $\rho_1$-symmetric periodic orbits occurs.
Hence we must have $x_{d_{\tau}}= y_{\tau}$ which is a contradiction.
It follows from the invariance of $F$ under $\rho_1$ that there has to  be another $\rho_2$-symmetric periodic orbit $ (\rho_1 \circ x_{d_\tau} ,2\sigma _\tau )$ for ${\tau \in (0, \delta)}$     by applying the $\rho_1$-symmetry. 

Assume  that all Reeb chords   with images in $\Sigma_{\tau} \cap U_1$ and with image in $U_2$ for $\tau  \in (0, \delta)$ are non-degenerate,  
 that $\Sigma_{\tau}$ is compact for all $\tau \in (0, \delta)$, 
 and that $\mu( c_{2,\tau}, 2\eta_{\tau}) =k$ for $\tau <0$ and $\mu(c_{2,\tau}, 2\eta_{\tau}) = k+p$ for $\tau >0$, where $p \in [-\dim N +1, \dim N -1] \setminus \{0\}$.
 We in particular have \eqref{eq:RKFfhormula}. 
 In the following we only consider the case that $p$ is positive. The case that $p$ is negative can be proved in a similar way. 
 As in the proof of Theorem \ref{thm: intromainthm},  there must be  a family of Reeb chords $\{(d_{\tau}^0, \sigma_{\tau}^0)\}_{\tau \in (0, \delta)}$ such that $\mu(d_{\tau}^0, \sigma_{\tau}^0) = k$ for all $\tau \in (0, \delta)$.
 In view of the presence of the $\rho_1$-symmetry, we obtain the family $\{(\rho_1 \circ d_{\tau}^0, \sigma_{\tau}^0)\}_{\tau \in (0, \delta)}$ whose elements have index equal to $k$ as well.
 Since the local eLRFH has a single generator in degree $k$, there must be  another family $\{(d_{\tau}^1, \sigma_{\tau}^1)\}_{\tau \in (0, \delta)}$ with index $k+1$
 and again by the $\rho_1$-symmetry, we also find the family $\{(\rho_1 \circ d_{\tau}^1, \sigma_{\tau}^1)\}_{\tau \in (0, \delta)}$ with index $k+1$.
 Repeating this argument, for each $N \in [k, k+p-1]$, we find two families of Reeb chords with index $N$ which are related by the $\rho_1$-symmetry. 
 Consider the families $\{( d_{\tau}^{p-1}, \sigma_{\tau}^{p-1})\}_{\tau \in (0, \delta)}$ and $\{(\rho_1 \circ d_{\tau}^{p-1}, \sigma_{\tau}^{p-1})\}_{\tau \in (0, \delta)}$ with index $k+p-1$.
 One of them kills one of two families of Reeb chords with index $k+p-2$ and the other is killed by $(c_{2,\tau}, 2\eta_{\tau})$ which has index $k+p$, see Figure \ref{fig:cor}.

The proofs for the remaining assertions are almost   identical to the proofs of the corresponding assertions in Theorem \ref{thm: intromainthm}.
 \end{proof}
 
  \begin{figure}[h]
\begin{center}
\begin{tikzpicture}[scale=0.8]
   \draw [fill] (0 ,0 ) circle [radius=0.05];
      \draw [fill] (0 ,-1 ) circle [radius=0.05];
   \draw [fill] (3 ,-1) circle [radius=0.05];
   \draw [fill] (0 ,-4 ) circle [radius=0.05];
   \draw [fill] (3 ,-4 ) circle [radius=0.05];
   \draw [fill] (0 ,-5 ) circle [radius=0.05];
   \draw [fill] (3 ,-5 ) circle [radius=0.05];
      \draw [fill] (0 ,-6 ) circle [radius=0.05];
   \draw [fill] (3 ,-6 ) circle [radius=0.05];
%%%%%%%%%%%%%%%%%%%%%%%
\draw[->] (0,0) to (1.5, -0.5);
\draw (3,-1) to (1.5, -0.5);
\draw[->] (0,-4) to (1.5, -4.5);
\draw (3,-5) to (1.5, -4.5);
\draw[->] (0,-5) to (1.5, -5.5);
\draw (3,-6) to (1.5, -5.5);
 \draw (0,-1) to (1, -1.3333);
\draw (3, -4) to (2,-3.6667);
%%%%%%%%%%%%%%%%%%%%%
   \draw [fill] (1.5 ,-1.9 ) circle [radius=0.02];
   \draw [fill] (1.5 ,-3.1 ) circle [radius=0.02];
   \draw [fill] (1.5 ,-2.2 ) circle [radius=0.02];
   \draw [fill] (1.5 ,-2.5 ) circle [radius=0.02];
   \draw [fill] (1.5 ,-2.8 ) circle [radius=0.02];
%%%%%%%%%%%%%%%%%%%%%
\node at (0,0)[left] {$ (c_{2, \tau}, 2\eta_{\tau})$};
\node at (0,-1)[left] {$ (d_{ \tau}^{p-1},  \sigma^{p-1}_{\tau})$};
\node at (3,-1)[right] {$ (\rho_1\circ d_{ \tau}^{p-1},  \sigma^{p-1}_{\tau})$};
\node at (0,-4)[left] {$ (d_{ \tau}^{ 2},  \sigma^{ 2}_{\tau})$};
\node at (3,-4)[right] {$ (\rho_1\circ d_{ \tau}^{ 2},  \sigma^{ 2}_{\tau})$};
\node at (0,-5)[left] {$ (d_{ \tau}^{ 1},  \sigma^{ 1}_{\tau})$};
\node at (3,-5)[right] {$ (\rho_1\circ d_{ \tau}^{ 1},  \sigma^{ 1}_{\tau})$};
\node at (0,-6)[left] {$ (d_{ \tau}^{0},  \sigma^{0}_{\tau})$};
\node at (3,-6)[right] {$ (\rho_1\circ d_{ \tau}^{0},  \sigma^{0}_{\tau})$};
%%%%%%%%%%%%%%%%%%%
\node at (-4, 0.8) {index};
\node at (1.5, 0.8) {Reeb chords};
\node at (-4, 0) {$k+p$};
\node at (-4, -1) {$k+p-1$};
\node at (-4, -4) {$k+2$};
\node at (-4, -5) {$k+1$};
\node at (-4, -6) {$k $}; 
%%%%%%%%%%%%%%%%%%%%%
   \draw [fill] (-4 ,-1.9 ) circle [radius=0.02];
   \draw [fill] (-4 ,-3.1 ) circle [radius=0.02];
   \draw [fill] (-4 ,-2.2 ) circle [radius=0.02];
   \draw [fill] (-4 ,-2.5 ) circle [radius=0.02];
   \draw [fill] (-4 ,-2.8 ) circle [radius=0.02];
%%%%%%%%%%%%%%%%%%%%%
\end{tikzpicture}
\end{center}
\caption{An illustration of  the proof of Corollary \ref{thm: intromain}. The arrows denote the boundary operator.}
\label{fig:cor}
\end{figure}
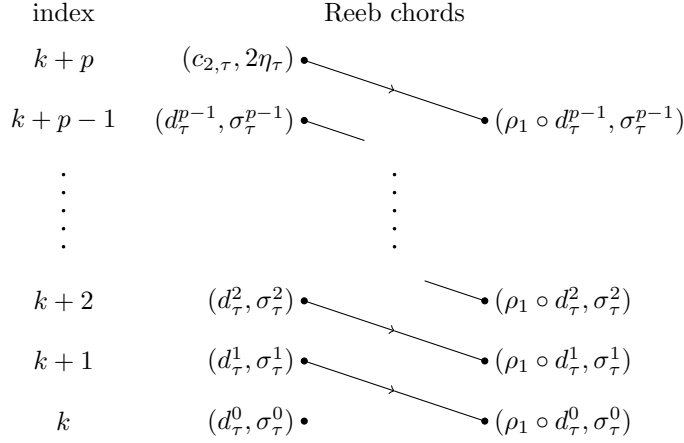

\begin{proof}[Proof of Theorem \ref{thm: intromainthm2}]
Arguing as above, if   $  (x_{c_{\tau}}, 2\eta_{\tau})  $ were the unique symmetric periodic orbit with image  in  $U_1$ and with lengths in $U_2$ for  all  $\tau>0$ small enough,  then
the local eLRFH would be a one-dimensional vector space. 
But for $\tau<0$ the local eLRFH is either trivial or   two-dimensional. 
This contradiction shows the existence of an additional family $\{(x_{d_{\tau}}, 2 \sigma_{\tau})\}_{\tau \in (0, \delta)}$ of symmetric periodic orbits for some $\delta>0$ small enough, which proves the first assertion.
The remaining proof  is  almost  identical to  the proof of Theorem  \ref{thm: intromainthm}.  This completes the proof of the theorem.
\end{proof}

\begin{proof}[Proof of Theorem \ref{thm: intromainthm1}]
As in the previous proof, we only prove the first assertion. 
If $ (x_{c_{\tau}}, 2\eta_{\tau}) $ were the unique symmetric periodic orbit  with image  in in $U_1$ and with lengths in $U_2$  for each $\tau<0$ close enough to $0$, then
the local eLRFH would be a one-dimensional vector space.
But this is not the case since  the third hypothesis on the family $\{  (x_{c_{\tau}}, 2\eta_{\tau}) \}_{\tau \in (-\epsilon, 0]}$ shows that the local eLRFH is trivial for $\tau >0$. 
Therefore, we find another family $\{(x_{d_{\tau}},2 \sigma_{\tau} )\}_{\tau \in (-\epsilon' ,0]}$ of $\rho$-symmetric periodic orbits for some $\epsilon' \in (0, \epsilon]$ such that $(x_{d_0} , 2\sigma_0 ) = (x_{c_0}, 2\eta_0)$.  
This finishes the proof of the theorem.
\end{proof}

 \bibliographystyle{abbrv}
\bibliography{mybibfile}

\begin{thebibliography}{10}

\bibitem{AM18}
A.~Abbondandolo and W.~J. Merry.
\newblock Floer homology on the time-energy extended phase space.
\newblock {\em J. Symplectic Geom.}, 16(2):279--355, 2018.

\bibitem{APS08}
A.~Abbondandolo, A.~Portaluri, and M.~Schwarz.
\newblock The homology of path spaces and {F}loer homology with conormal
  boundary conditions.
\newblock {\em J. Fixed Point Theory Appl.}, 4(2):263--293, 2008.

\bibitem{AS06}
A.~Abbondandolo and M.~Schwarz.
\newblock On the {F}loer homology of cotangent bundles.
\newblock {\em Comm. Pure Appl. Math.}, 59(2):254--316, 2006.

\bibitem{RKPCZ}
P.~Albers, J.~W. Fish, U.~Frauenfelder, and O.~van Koert.
\newblock The {C}onley-{Z}ehnder indices of the rotating {K}epler problem.
\newblock {\em Math. Proc. Cambridge Philos. Soc.}, 154(2):243--260, 2013.

\bibitem{AF09negative}
P.~Albers and U.~Frauenfelder.
\newblock Floer homology for negative line bundles and {R}eeb chords in
  prequantization spaces.
\newblock {\em J. Mod. Dyn.}, 3(3):407--456, 2009.

\bibitem{AFvKP12}
P.~Albers, U.~Frauenfelder, O.~van Koert, and G.~P. Paternain.
\newblock Contact geometry of the restricted three-body problem.
\newblock {\em Comm. Pure Appl. Math.}, 65(2):229--263, 2012.

\bibitem{AmanZehnder}
H.~Amann and E.~Zehnder.
\newblock Nontrivial solutions for a class of nonresonance problems and
  applications to nonlinear differential equations.
\newblock {\em Ann. Scuola Norm. Sup. Pisa Cl. Sci. (4)}, 7(4):539--603, 1980.

\bibitem{Asselle16}
L.~Asselle.
\newblock On the existence of {E}uler-{L}agrange orbits satisfying the conormal
  boundary conditions.
\newblock {\em J. Funct. Anal.}, 271(12):3513--3553, 2016.

\bibitem{Ad20}
C.~Aydin.
\newblock {\em private communication}, 2020.

\bibitem{BFvK19}
E.~Belbruno, U.~Frauenfelder, and O.~van Koert.
\newblock The omega limit set of a family of chords.
\newblock {\em J. Topol. Anal.}, online ready.

\bibitem{CFHW}
K.~Cieliebak, A.~Floer, H.~Hofer, and K.~Wysocki.
\newblock Applications of symplectic homology. {II}. {S}tability of the action
  spectrum.
\newblock {\em Math. Z.}, 223(1):27--45, 1996.

\bibitem{CFvK14Finsler}
K.~Cieliebak, U.~Frauenfelder, and O.~van Koert.
\newblock The {F}insler geometry of the rotating {K}epler problem.
\newblock {\em Publ. Math. Debrecen}, 84(3-4):333--350, 2014.

\bibitem{CF09}
K.~Cieliebak and U.~A. Frauenfelder.
\newblock A {F}loer homology for exact contact embeddings.
\newblock {\em Pacific J. Math.}, 239(2):251--316, 2009.

\bibitem{CV98}
M.~C. Ciocci and A.~Vanderbauwhede.
\newblock Bifurcation of periodic orbits for symplectic mappings.
\newblock {\em J. Differ. Equations Appl.}, 3(5-6):485--500, 1998.

\bibitem{CV04}
M.-C. Ciocci and A.~Vanderbauwhede.
\newblock Bifurcation of periodic points in reversible diffeomorphisms.
\newblock In {\em Proceedings of the {S}ixth {I}nternational {C}onference on
  {D}ifference {E}quations}, pages 75--93. CRC, Boca Raton, FL, 2004.

\bibitem{CZ}
C.~Conley and E.~Zehnder.
\newblock Morse-type index theory for flows and periodic solutions for
  {H}amiltonian equations.
\newblock {\em Comm. Pure Appl. Math.}, 37(2):207--253, 1984.

\bibitem{Cont06}
G.~Contreras.
\newblock The {P}alais-{S}male condition on contact type energy levels for
  convex {L}agrangian systems.
\newblock {\em Calc. Var. Partial Differential Equations}, 27(3):321--395,
  2006.

\bibitem{DX}
Y.~Deng and Z.~Xia.
\newblock Conley-{Z}ehnder index and bifurcation of fixed points of
  {H}amiltonian maps.
\newblock {\em Ergodic Theory Dynam. Systems}, 38(6):2086--2107, 2018.

\bibitem{DE76}
R.~L. Devaney.
\newblock Reversible diffeomorphisms and flows.
\newblock {\em Trans. Amer. Math. Soc.}, 218:89--113, 1976.

\bibitem{FS16RFH}
U.~Frauenfelder and F.~Schlenk.
\newblock {$S^1$}-equivariant {R}abinowitz-{F}loer homology.
\newblock {\em Hokkaido Math. J.}, 45(3):293--323, 2016.

\bibitem{book}
U.~Frauenfelder and O.~van Koert.
\newblock {\em The restricted three body problem and holomorphic curves}.
\newblock Pathways in Mathematics. Birkh\"auser Basel, 1 edition, 2018.

\bibitem{GinzGurel}
V.~L. Ginzburg and B.~Z. G\"{u}rel.
\newblock Local {F}loer homology and the action gap.
\newblock {\em J. Symplectic Geom.}, 8(3):323--357, 2010.

\bibitem{GMVF81}
J.~M. Greene, R.~S. MacKay, F.~Vivaldi, and M.~J. Feigenbaum.
\newblock Universal behaviour in families of area-preserving maps.
\newblock {\em Phys. D}, 3(3):468--486, 1981.

\bibitem{Henon}
M.~H\'enon.
\newblock Numerical exploration of the restricted problem. {V}. {H}ill's case:
  periodic orbits and their stability.
\newblock {\em Astron. Astrophs.}, 1:223--238, 1969.

\bibitem{Hill}
G.~W. Hill.
\newblock Researches in the lunar theory.
\newblock {\em Amer. J. Math.}, 1(1):5--26, 129--147, 1878.

\bibitem{HZbook}
H.~Hofer and E.~Zehnder.
\newblock {\em Symplectic invariants and {H}amiltonian dynamics}.
\newblock Modern Birkh\"{a}user Classics. Birkh\"{a}user Verlag, Basel, 2011.
\newblock Reprint of the 1994 edition.

\bibitem{KKJplus}
J.~Kim and S.~Kim.
\newblock ${J}^+$-like invariants of periodic orbits of the second kind in the
  restricted three-body problem.
\newblock {\em J. Topol. Anal.}, online ready.

\bibitem{LR98}
J.~S.~W. Lamb and J.~A.~G. Roberts.
\newblock Time-reversal symmetry in dynamical systems: a survey.
\newblock {\em Phys. D}, 112(1-2):1--39, 1998.
\newblock Time-reversal symmetry in dynamical systems (Coventry, 1996).

\bibitem{JLee}
J.~Lee.
\newblock Spectral invariant of {F}loer homology and its application to
  {H}ill's lunar problem, {P}h.{D}. thesis, {S}eoul {N}ational {U}niversity.

\bibitem{LT09}
M.~F.~S. Lima and M.~A. Teixeira.
\newblock Families of periodic orbits in resonant reversible systems.
\newblock {\em Bull. Braz. Math. Soc. (N.S.)}, 40(4):511--537, 2009.

\bibitem{Mclean}
M.~McLean.
\newblock Local {F}loer homology and infinitely many simple {R}eeb orbits.
\newblock {\em Algebr. Geom. Topol.}, 12(4):1901--1923, 2012.

\bibitem{Me14}
W.~J. Merry.
\newblock Lagrangian {R}abinowitz {F}loer homology and twisted cotangent
  bundles.
\newblock {\em Geom. Dedicata}, 171:345--386, 2014.

\bibitem{RSindex}
J.~Robbin and D.~Salamon.
\newblock The {M}aslov index for paths.
\newblock {\em Topology}, 32(4):827--844, 1993.

\bibitem{V86}
A.~Vanderbauwhede.
\newblock Bifurcation of subharmonic solutions in time-reversible systems.
\newblock {\em Z. Angew. Math. Phys.}, 37(4):455--478, 1986.

\bibitem{V92}
A.~Vanderbauwhede.
\newblock Branching of periodic solutions in time-reversible systems.
\newblock In {\em Geometry and analysis in nonlinear dynamics ({G}roningen,
  1989)}, volume 222 of {\em Pitman Res. Notes Math. Ser.}, pages 97--113.
  Longman Sci. Tech., Harlow, 1992.

\bibitem{Kaz13}
K.~Yagasaki.
\newblock Bifurcations from one-parameter families of symmetric periodic orbits
  in reversible systems.
\newblock {\em Nonlinearity}, 26(5):1345--1360, 2013.

\end{thebibliography}

\end{document}